\documentclass[12pt,reqno]{amsart}
\usepackage{amssymb,color}

\newcommand{\la}{\lambda}
\newcommand{\al}{\alpha}

\newcommand{\fy}{\varphi}
\newcommand{\pd}{\partial}
\newcommand{\p}{\partial}
\newcommand{\I}{\infty}

\newcommand{\ti}{\widetilde}
\newcommand{\R}{\mathbb{R}}
\newcommand{\C}{\mathbb{C}}
\newcommand{\N}{\mathbb{N}}
\newcommand{\Z}{\mathbb{Z}}
\newcommand{\F}{\mathcal{F}}
\renewcommand{\S}{\mathbb{S}}
\renewcommand{\Re}{\mathop{\mathrm{Re}}}
\renewcommand{\Im}{\mathop{\mathrm{Im}}}
\renewcommand{\bar}{\overline}
\renewcommand{\hat}{\widehat}

\renewcommand{\r}{\rho}

\numberwithin{equation}{section}

\newtheorem{thm}{Theorem}[section]
\newtheorem{cor}[thm]{Corollary}
\newtheorem{lem}[thm]{Lemma}

\theoremstyle{remark}
\newtheorem{rem}{Remark}

\newcommand{\ran}{\rangle}
\newcommand{\lan}{\langle}

\newcommand{\lec}{\lesssim}
\newcommand{\gec}{\gtrsim}

\newcommand{\EQ}[1]{\begin{equation} \begin{split} #1
 \end{split} \end{equation}}

\newcommand{\Del}[1]{}
\newcommand{\CAS}[1]{\begin{cases} #1 \end{cases}}
\newcommand{\pt}{&}
\newcommand{\pr}{\\ &}
\newcommand{\pq}{\quad}
\newcommand{\pn}{}

\newcommand{\prQ}{\\ &\qquad}

\newcommand{\LR}[1]{{\lan #1 \ran}}
\newcommand{\de}{\delta}

\renewcommand{\t}{\tau}
\newcommand{\be}{\beta}

\newcommand{\ga}{\gamma}
\newcommand{\x}{\xi}

\newcommand{\s}{\sigma}

\newcommand{\na}{\nabla}

\renewcommand{\th}{\theta}

\newcommand{\supp}{\operatorname{supp}}

\newcommand{\D}{\mathcal{D}}
\newcommand{\E}{\mathcal{E}}

\newcommand{\De}{\Delta}

\newcommand{\PX}[2]{\big\langle #1 \big| #2 \big\rangle_x}
\newcommand{\IN}[1]{\text{ in }#1}

\newcommand{\mat}[1]{\begin{pmatrix} #1 \end{pmatrix}}

\newcommand{\M}[1]{\vec{#1}}
\newcommand{\V}[1]{{\bf #1}}
\newcommand{\Sc}[1]{{\mathcal{#1}}}

\newcommand{\vdt}{\cdot}
\newcommand{\cdt}{\,\overset{{}_\circ}{}\,}
\newcommand{\res}{\check} 
\newcommand{\hmu}{\M{h}[\mu]}
\newcommand{\rs}[1]{{#1}^{\not s}}
\newcommand{\Si}{\Sigma}
\newcommand{\Tr}[1]{{}^t\!{#1}}

\newcommand{\ip}[2]{(#1 \mid #2)}

\begin{document}
\title[Stability and oscillation of harmonic maps]
{Asymptotic stability, concentration, and \\
oscillation in harmonic map heat-flow, \\
Landau-Lifshitz, and Schr\"odinger maps on $\R^2$}

\author{Stephen Gustafson}
\author{Kenji Nakanishi}
\author{Tai-Peng Tsai}

\begin{abstract}
We consider the Landau-Lifshitz equations of ferromagnetism 
(including the harmonic map heat-flow and Schr\"odinger flow 
as special cases) for degree $m$ equivariant maps from $\R^2$ 
to $\S^2$. If $m \geq 3$, we prove that near-minimal energy 
solutions converge to a harmonic map as $t \to \I$ ({\it asymptotic 
stability}), extending previous work \cite{GKT2} down to degree $m = 3$. 
Due to slow spatial decay of the harmonic map components, a new 
approach is needed for $m=3$, involving (among other tools) a ``normal 
form'' for the parameter dynamics, and the 2D radial double-endpoint 
Strichartz estimate for Schr\"odinger operators with sufficiently 
repulsive potentials (which may be of some independent interest). 
When $m=2$ this asymptotic stability may fail: in the 
case of heat-flow with a further symmetry restriction, 
we show that more exotic asymptotics are possible, including 
infinite-time concentration (blow-up), and even ``eternal oscillation''. 
\end{abstract}

\maketitle

\tableofcontents

\section{Introduction and Results}

The {\it Landau-Lifshitz} (sometimes {\it Landau-Lifshitz-Gilbert})
equation describing the dynamics of an 2D isotropic ferromagnet 
is (eg. \cite{IK})
\EQ{  \label{LL}
  \M{u}_t = a_1 (\De \M{u} + |\nabla \M{u}|^2 \M{u})
  + a_2 \M{u} \times \De \M{u},  
  \quad\quad a_1 \geq 0, \;\; a_2 \in \R }
where the {\it magnetization vector} 
$\M{u} = \M{u}(t,x) = (u_1, u_2, u_2)$
is a $3$-vector with normalized length, so can be considered
a map into the $2$-sphere $\S^2$:
\EQ{
  \M{u} : [0,T) \times \R^2 \to
  \S^2 := \{ \M{u} \in \R^3 \; | \; |\M{u}| = 1 \}. }
The special case $a_2 = 0$ of~\eqref{LL} is the very well-studied 
{\it harmonic map heat-flow} into $\S^2$, while
the special case $a_1 = 0$ is known as the {\it Schr\"odinger flow}
(or {\it Schr\"odinger map}) equation, the geometric generalization
of the linear Schr\"odinger equation for maps into the K\"ahler
manifold $\S^2$.  

In order to exhibit the simple geometry of~\eqref{LL} more clearly, 
we introduce, for $\M{u} \in \S^2$, the tangent space
\EQ{
  T_{\M{u}} \S^2 := \M{u}^\perp =
  \{ \M{\xi} \in \R^3 \; | \; \M{u} \cdot \M{\xi} = 0 \} }
to the sphere $\S^2$ at $\M{u}$.
For any vector $\M{v} \in \R^3$, we define two operations 
on vectors:
\EQ{
  J^{\M{v}} := \M{v} \times, \quad\quad
  P^{\M{v}} := -J^{\M{v}} J^{\M{v}}. }
For $\M{u} \in \S^2$, $P^{\M{u}}$ projects vectors orthogonally
onto $T_{\M{u}} \S^2$, while $J^{\M{u}}$ is a $\pi/2$ rotation
(complex structure) on $T_{\M{u}} \S^2$.   
Denoting 
\EQ{
  a = a_1 + i a_2 \in \C, }
the Landau-Lifshitz equation~\eqref{LL} may be written
\EQ{  \label{LL2}
  \M{u}_t = P^{\M{u}}_a \De \M{u}, \quad\quad
  P^{\M{u}}_a := a_1 P^{\M{u}} + a_2 J^{\M{u}} }
The energy associated to~\eqref{LL} is simply the 
Dirichlet functional
\EQ{
  \E(\M{u}) = \frac 12 \int_{\R^2} |\nabla \M{u}|^2 dx }
and~\eqref{LL2} formally yields the energy identity
\EQ{
    \E(\M{u}(t)) + 2 a_1 \int_0^t \int_{\R^2} 
    | P^{\M{u}} \De \M{u} (s,x) |^2 dx ds  = \E(\M{u}(0)) }
implying, in particular, energy non-increase if $a_1 > 0$,
and energy conservation if $a_1 = 0$ (Schr\"odinger map).

To a finite-energy map $\M{u} : \R^2 \to \S^2$ is associated
the {\it degree}
\EQ{ \label{deg}
  \deg(\M{u}) := \frac{1}{4\pi} \int_{\R^2} \M{u}_{x_1} \cdot 
  J^{\M{u}} \M{u}_{x_2} dx. } 
If $\lim_{|x| \to \infty} \M{u}(x)$ exists (which will be the 
case below), we may identify $\M{u}$ with a map $\S^2 \to \S^2$, 
and if the map is smooth, $\deg(\M{u})$ is the usual
Brouwer degree (in particular, an integer).  
It follows immediately from expression~\eqref{deg} 
that the energy is bounded from below by the degree:
\EQ{  \label{enbound}
    \E(\M{u}) =  \frac 12 \int_{\R^2} |\M{u}_{x_1}
    - J^{\M{u}} \M{u}_{x_2} |^2 + 4\pi \deg(\M{u})
    \geq 4\pi \deg(\M{u}), }
and equality here is achieved exactly at {\it harmonic maps}
solving the first-order equations
\EQ{  \label{cr}
  \M{u}_{x_1} = J^{\M{u}} \M{u}_{x_2} } 
which, in stereographic coordinates
\EQ{  \label{stereo}
  \S^2 \ni \M{u} \;\; \longleftrightarrow \;\;
  \frac{u_1 + iu_2}{1-u_3} \in \C \cup \{ \I \} }
are the Cauchy-Riemann equations, and the solutions are 
rational functions. 
These harmonic maps are critical points 
of the energy $\E$ and, in particular, static solutions of the 
Landau-Lifshitz equation~\eqref{LL}.

In this paper we specialize to the class of 
$m$-{\it equivariant maps}, for some $m \in \Z^+$:
\EQ{  \label{equiv}
  \M{u}(t,x) = e^{m\th R} \M{v}(t,r), \quad\quad
  \M{v} : [0,T) \times [0,\infty) \to \S^2 }
with notations
\EQ{
  R := J^{\M{k}} = \M{k}\times, \quad \M{k} = (0,0,1), }
and polar coordinates
\EQ{ 
  x_1+ix_2=re^{i\th}. }
In terms of the radial profile map 
$\M{v} = (v_1, v_2, v_3)$, the energy is
\EQ{
  \E(\M{u}) = \pi \int_0^\infty \left(
  | \M{v}_r |^2 + \frac{m^2}{r^2}(v_1^2 + v_2^2)
  \right) r dr  .}
Finite energy implies $\M{v}$ is continuous in $r$ and
$\lim_{r \to 0} \M{v} = \pm \M{k}$, 
$\lim_{r \to \I} \M{v} = \pm \M{k}$ 
(see~\cite{GKT1} for details).
We force non-trivial topology by working in the class of maps
\EQ{
  \Si_m := \{ \M{u} = e^{m \th R} \M{v}(r) \; | \; 
  \E(\M{u}) < \infty, \; \M{v}(0)=-\M{k},\ \M{v}(\I)=\M{k} \}. }
It is easy to check that the degree of such maps is $m$:
\EQ{  \label{degm}
   \deg \restriction_{\Si_M} \; \equiv m.  }

The harmonic maps saturating inequality~\eqref{enbound} which 
also lie in $\Sigma_m$ are those corresponding to 
$\be z^m$ ($\be\in \C^\times=\C \backslash \{ 0 \}$) 
in stereographic coordinates~\eqref{stereo}.
In the representation $\S^2 \subset \R^3$, the harmonic map 
corresponding to $z^m$ is given by
\EQ{
 e^{m \theta R} \M{h}(r), \pq
 \M{h} = (h_1,0,h_3), \pq 
 h_1 = \frac{2}{r^m+r^{-m}}, \pq 
 h_3 = \frac{r^m-r^{-m}}{r^m+r^{-m}}.}
The full two-dimensional family of $m$-equivariant harmonic maps 
in $\Si_m$ is then generated by rotation and scaling, so for
$s > 0$ and $\al \in \R$, we denote   
\EQ{
 \mu = m\log s + i\al, \pq \hmu = e^{\al R} \M{h}^s, \pq
 \M{h}^s=\M{h}(r/s).}
The harmonic map
$e^{m \theta R} h[\mu]$ corresponds 
under stereographic projection to $e^{-\bar\mu} z^m$. 

We are concerned here with basic global properties of solutions
of the Landau-Lifshitz equations~\eqref{LL}, especially
the possible formation of singularities, and the 
long-time asymptotics. 

For finite-energy solutions of~\eqref{LL} in 2 space dimensions, finite-time singularity formation is only known 
to occur in the case of the $1$-equivariant harmonic map 
heat-flow ($a_2 = 0$) -- the first such result \cite{CDY} 
was for the problem on a disk with Dirichlet boundary 
conditions (this was extended to $\Sigma_1$ on $\R^2$ in 
\cite{GGT}). Examples of finite-time blow-up for different
target manifolds (not the physical case $\S^2$) are also
known (eg. \cite{Top2}).
    
For the Schr\"odinger case ($a_1 = 0$), \cite{CSU} showed that
small-energy solutions remain regular. 
In the present setting, the energy is not small -- indeed
by~\eqref{enbound} and~\eqref{degm}, 
\EQ{
  \E \restriction_{\Si_m} \; \geq 4 \pi m. }    
A self-similar blow-up solution, which however carries
infinite energy, is constructed in~\cite{GSZ}. 

In the recent works \cite{GKT2,GGT,GGKT}, it was shown that
when $m \geq 4$, solutions of~\eqref{LL} in $\Si_m$
with near minimal energy ($\E(\M{u}) \approx 4 \pi m$)
are globally regular, and converge asymptotically to 
a member $e^{m \th R} \hmu$ of the harmonic map family.
In particular, the harmonic maps are {\it asymptotically 
stable}. The analysis there fails to extend to $m \leq 3$, 
due to the slower spatial decay of $\frac{d}{d \mu} \M{h}[\mu]$
(a point which we hope to clarify below). With a new
approach, we can now handle the case $m=3$ as well: 

\begin{thm}  \label{thm1}
Let $m\ge 3$, $a = a_1 + i a_2 \in \C \backslash \{0\}$, 
and $a_1 \ge 0$. 
Then there exists $\de>0$ such that for any
$\M{u}(0,x) \in \Sigma_m$ with $\E(\M{u}(0))\le 4m\pi+\de^2$, 
we have a unique global solution $\M{u}\in C([0,\I);\Si_m)$
of~\eqref{LL}, satisfying 
$\na\M{u}\in L^2_{t,loc}([0,\I);L^\I_x)$. 
Moreover, for some $\mu \in \C$ we have 
\EQ{
  \|\M{u}(t)-e^{m\th R}\hmu\|_{L^\I_x} + a_1
  \E(\M{u}(t)-e^{m\th R}\hmu) \to 0 \pq \mbox{ as } \; t\to\I. }
\end{thm}
In short, every solution with energy close to the minimum converges to one of the harmonic maps uniformly in $x$ as $t\to\I$. 
Even for the higher degrees $m\ge 4$, this result is stronger than the 
previous ones \cite{GKT2,GGT,GGKT}, where the convergence was given 
only in time average.\footnote{The statements in the previous papers 
do not follow directly from Theorem~\ref{thm1}, but are implied 
by the proof in this paper.}
Note that in the dissipative case ($a_1 > 0$), solutions 
converge to a harmonic map also in the energy norm, while this is
impossible for the conservative Schr\"odinger flow ($a_1=0$).

The analysis for the case $m = 2$ seems trickier still, 
and we have results only in special case of the harmonic 
map heat-flow ($a_2 = 0$) with the further restriction that the 
image of the solution remain on a great circle: $v_2 \equiv 0$
(a condition preserved by the evolution only for the heat-flow). 
These results show, in particular, that the strong asymptotic 
stability result of Theorem~\ref{thm1} for $m \geq 3$ is no 
longer valid; instead, more exotic asymptotics are possible, 
including infinite-time concentration (blow-up) and 
``eternal oscillation'': 

\begin{thm}  \label{thm2}
Let $m=2$ and $a > 0$. 
Then there exists $\de>0$ such that for any
$\M{u}(0,x) = e^{2 \th R} \M{v}(0,r) \in \Sigma_2$ 
with $\E(\M{u}(0))\le 8\pi+\de^2$, and $v_2(0,r) \equiv 0$, 
we have a unique global solution $\M{u} \in C([0,\I);\Si_2)$
satisfying $\na\M{u}\in L^2_{t,loc}([0,\I);L^\I_x)$. 
Moreover, for some continuously differentiable
$s:[0,\I)\to(0,\I)$ we have 
\EQ{ \label{conv}
  \|\M{u}(t)-e^{m\th R}\M{h}(r/s(t))\|_{L^\I_x} +
  \E(\M{u}(t)-e^{m\th R}\M{h}(r/s(t))) \to 0 
  \;\; \mbox{ as } \; t \to \I. }
In addition, we have the following asymptotic formula for $s(t)$:
\EQ{  \label{s asy} 
  (1+o(1))\log(s(t)) = \frac{2}{\pi}\int_{1}^{\sqrt{at}} 
  \frac{v_1(0,r)}{r} dr + O_c(1),}
where as $t \to \I$, $o(1) \to 0$ and $O_c(1)$ converges to 
some finite value. In particular there are initial data 
yielding each of the following types of asymptotic behavior:
\begin{enumerate}
\item $s(t) \to \exists \; s_\I \in (0, \I)$. 
\item $s(t)\to 0$.
\item $s(t)\to \I$. 
\item $0=\liminf s(t)<\limsup s(t)<\I$.
\item $0<\liminf s(t)<\limsup s(t)=\I$.
\item $0=\liminf s(t)<\limsup s(t)=\I$.
\end{enumerate}
\end{thm}
Estimate~\eqref{conv} shows that these solutions 
do converge asymptotically to the {\it family} of harmonic maps.
However, the evolution along this family, described by the 
parameter $s(t)$, does not necessarily approach a particular map 
in $\Si_2$ (although it might -- case (1)). The solution may 
in fact converge pointwise (but not uniformly) to a constant map 
$\pm \M{k}$ (which has zero energy, zero degree, and lies outside
$\Si_2$) as in (2)-(3) (this is infinite-time blow-up or 
concentration), or it may asymptotically ``oscillate'' along  
the harmonic map family, as in (4)-(6). 

Note that the above classification (1)-(6) is stable against initial ``local" perturbation. 
Namely, if two initial data $v^1(0)$ and $v^2(0)$ satisfy 
\EQ{
 \int_1^\I \frac{|v_1^1(0,r)-v_1^2(0,r)|}{r} dr < \I,}
the corresponding solutions have the same asymptotic type among (1)-(6). 
More precisely, the difference of their scaling parameters converges in $(0,\I)$.   
The point is that the energy just barely fails to control the above integral. 

In particular, the oscillatory behavior in (4)-(6) is driven solely by the distribution around spatial infinity. 
In fact, if we replace the domain $\R^2$ by the disk $D=\{x\in\R^2\mid|x|< 1\}$ with the same symmetry restriction with $m=2$ and the same boundary conditions $v(t,0)=-\M{k}$ and $v(t,1)=\M{k}$, then it is known \cite{AH,GS} that all the solutions behave like (2), namely they concentrate at $x=0$ as $t\to\I$. 
Also, if we replace the domain $\R^2$ by $\S^2$, then we can rather easily show in the dissipative case $a_1>0$ that the solution converges to one harmonic map for all $m\in\N$, by the argument in this paper, or even those in the previous papers. 
We state the result on $\S^2$ in Appendix A with a sketch of the proof. 

As for other target spaces, we should mention an example in \cite[Section 5]{Top}, which is ``eternal winding" by heat flow around a compact 1-parameter family of harmonic maps from $\S^2$ to $\S^2\times\R^2$ with an artificial warped metric. 
Our result has the following comparative advantages:  
\begin{enumerate}
\item The geometric setting is very simple and physically natural. 
\item The asymptotic formula is explicit in terms of the initial data. 
\end{enumerate}
In addition, our analysis works in the same way in the dissipative ($a_1>0$) and the dispersive ($a_1=0$) cases. 
We need $a_2=0$ in Theorem \ref{thm2} only because the angular parameter $\al(t)$ gets beyond our control (hence we remove it by the constraint), but the rest of our arguments could work in the general case.\footnote{We will use the parameter convergence in the proof of Theorem \ref{thm1} in the dispersive case $a_1=0$ to fix our linearized operator. 
However it is possible to treat the linearized operator even with non-convergent parameter and $a_1=0$, if we assume one more regularity on the initial data. 
We do not pursue it here since the wild behavior of $\al(t)$ prevents us from using it.}


\subsection{The main difficulty and the main idea}
The standard approach for asymptotic stability is to decompose the solution into a leading part with finite dimensional parameters varying in time, and the rest decaying in time either by dissipation or by dispersion. 
In our context, we want to decompose the solution in the form 
\EQ{
 \M{v}(t) = \M{h}[\mu(t)] + \res{v}(t)}
such that the remainder $\res{v}(t)$ decays, and the parameter 
$\mu(t) \in \C$ converges as $t\to\I$ (at least for Theorem \ref{thm1}). 
In favorable cases (the higher $m$, in our context), we can choose 
$\mu(t)$ such that all secular modes for $\res{v}(t)$ are absorbed 
into the time evolution of the main part $\M{h}[\mu(t)]$. 
This means that the kernel of the linearized operator for $\res{v}(t)$ is spanned by the parameter derivatives of $\M{h}[\mu]$, and hence we can put that component of $\p_t\M{v}(t)$ into $\p_t\M{h}[\mu(t)]$. 
This is good both for $\res{v}(t)$ and $\mu(t)$, because  
\begin{enumerate}
\item $\res{v}(t)$ will be free from secular modes, and so we can expect it to decay by dissipation or dispersion, at least at the linearized level.  
\item The decomposition is preserved by the linearized equation. Hence $\dot\mu(t)$ is affected by $\res{v}$ only superlinearly, i.e.~at most in quadratic terms. 
\end{enumerate}
In particular, if we can get $L^2$ decay of $\res{v}$ in time, then $\dot\mu(t)$ becomes integrable in time, and so converges as $t\to\I$. This is indeed the case for $m>3$. 

However, the above naive argument does not take into account the space-time behavior of each component. 
The problem comes from the fact that the decomposition and the decay estimate must be implemented in different function spaces, and they may be incompatible if the eigenfunctions decay too slowly at the spatial infinity. 

In fact, the parameter derivative of $\hmu$ is given by 
\EQ{
 d\hmu = h_1^s e^{\al R}[(h_3^s,0,-h_1^s)d\mu_1+(0,1,0)d\mu_2]}
and hence the eigenfunctions are $O(r^{-m})$ for $r\to\I$, i.e.~slower for lower $m$. 
On the other hand, the spatial decay property in the function space for the time decay estimate is essentially determined by the invariance of our problem under the scaling 
\EQ{
 \M{v}(t,x) \mapsto \M{v}(\la^2 t, \la x),}
which maps solutions into solutions, preserving the energy. 
If we want $L^2$ decay in time (so that we can integrate quadratic terms in $\dot\mu$), then a function space with the right  scaling is given by 
\EQ{
 \res{v}/r \in L^2_t L^\I_x.}
To preserve such norms in $x$ under the orthogonal projection, 
the eigenfunction must be in the dual space, for which $m>3$ is necessary. 
Indeed, this is the essential reason for the restriction $m\ge 4$ in the previous works \cite{GKT1,GKT2,GGT}. 
We emphasize that the above difficulty is common for the dissipative and dispersive cases, since they share the same scaling property. 
That is, the dissipation does not help with this issue, even though it 
gives us more flexibility in the form of decay estimates. 

The main novelty of the present approach is the non-orthogonal 
decomposition 
\EQ{
  L^2_x = (h_1^s) \oplus (\fy^s)^\perp,}
where $\fy^s(r)$ is smooth and supported away from $r=0$ and from $r=\I$, 
so that the (non-orthogonal) projection may preserve the decay estimates. 
This is good for the remainder $\res{v}$, but not for the parameter $\mu$
---the decomposition is no longer preserved by the linearized evolution, since they have no particular relation. 
This implies that we get a new error term in $\dot\mu(t)$ which is 
linear in $\res{v}(t)$ (see Section~\ref{parameter}). 
This contribution is handled by including it in 
a sort of ``normal form'' for the dynamics of the parameters 
$\mu(t)$, explained in Section~\ref{normal form}.
In particular, it is this new term which drives the non-trivial 
dynamics for the $m=2$ heat-flow given in Theorem~\ref{thm2}. 

For the purely dispersive (Schr\"odinger map) case, 
one tool we use should be of some independent interest: the 2D radial 
``double-endpoint Strichartz estimate'' for Schr\"odinger operators 
with sufficiently ``repulsive'' potentials (in the absence of a 
potential, the estimate is false). The proof is given in 
Section~\ref{end}.

\subsection{Organization of the paper}
In Section~\ref{frame}, we use the ``generalized 
Hasimoto transform'' to derive the main equation used
to obtain time-decay estimates of the remainder term.
Section~\ref{splitting} gives the details of the
solution decomposition described above, 
and addresses the inversion of the Hasimoto transform.
The estimates for going back and forth between the 
different coordinate systems (the ``Hasimoto'' one of 
Section~\ref{frame} and the decomposition of 
Section~\ref{splitting}) are given in Section~\ref{coord}.  
Section~\ref{decay} is devoted to establishing the time-decay 
(dispersive if $a_1=0$, diffusive if $a_1 > 0$) of the 
remainder term, using energy-, Strichartz-, and scattering-type 
estimates. The dynamics of the parameters $\mu(t)$ are
derived and estimated in Section~\ref{parameter}.
The leading term in the equation for $\dot \mu$ is not
integrable in time, and so Section~\ref{normal form} gives
an integration by parts in time to identify (and estimate)
a kind of ``normal form'' correction to $\mu(t)$, whose time
derivative is integrable. At this stage, the proof
of Theorem~\ref{thm1} for $m > 3$ is complete. A more subtle 
estimate of an error term for $m=3$ is done in 
Section~\ref{m=3}, completing the proof in that case.
Finally, in Section~\ref{m=2}, the normal form correction
is analyzed in the case $m=2$, $a_2 =0$, $v_2=0$, in order to 
prove Theorem~\ref{thm2}. Proofs of certain linear estimates
(including the double-endpoint Strichartz) are relegated    
to Section~\ref{pf lin ests}. Appendix A states the
analogous theorems for domain $\S^2$ and sketches the proofs.


\subsection{Some further notation}

We distinguish inner products in $\R^3$ and $\C$ by 
\EQ{
  \M{a} \vdt \M{b} = \sum_{k=1}^3 a_k b_k, \pq 
  a \cdt b = \Re a \Re b + \Im a \Im b.}
Both will be used for $\C^3$ vectors too. 
The $L^2_x$ inner-product is denoted
\EQ{
 \ip{f}{g} = \int_{\R^2} f(x)\bar{g(x)} dx,}
while $(f,g)$ just denotes a pair of functions. 
For any radial function $f(r)$ and any parameter $s>0$, we denote rescaled functions by 
\EQ{
  f^s(r) := f(r/s), \pq \rs{f}(r) := f(r/s)s^{-2}. }
We denote the Fourier transform on $\R^2$ by $\F$, and, for radial 
functions, the Fourier-Bessel transform of order $m$ by $\F_m$: 
\EQ{
 (\F f)(\x)= \frac{1}{2\pi} \int_{\R^2}f(x)e^{-ix \cdot \x}dx, \pq (\F_m)f(\r) = \int_0^\I J_m(r\r)f(r)rdr,}
where $J_m$ is the Bessel function of order $m$. For $m\in\Z$ we have 
\EQ{
 J_m(r) = \frac{1}{2\pi}\int_{-\pi}^{\pi} e^{im\th-ir\sin\th} d\th, \pq \F[f(r)e^{im\th}] = i^m(\F_mf)e^{im\th}.}
We denote the Laplacian $\De_x$ on the subspace spanned by $d$-dimensional spherical harmonics of order $m$ by 
\EQ{
 \De_d^{(m)} := \p_r^2 + (d-1)r^{-1}\p_r - m(m+d-2)r^{-2}.}
Finally, the space $L^p_q$ is the dyadic version of $L^p(rdr)$ 
defined by the norm 
\EQ{ \label{Lpq}
  \|f\|_{L^p_q} = \bigl\|\|f(r)\{2^j<r<2^{j+1}\}\|_{L^p(rdr)}\bigr\|_{\ell^q_j(\Z)}.}


\section{Generalized Hasimoto transform}
\label{frame}
First we recall from the previous papers \cite{GKT1,GGT,GGKT} the equation for the remainder part, which is written in terms of a derivative vanishing exactly on the harmonic maps, and so independent of the decomposition. 
The equation was originally derived in \cite{CSU} in the case of small 
energy solutions (hence with no harmonic map component), and called 
there the {\it generalized Hasimoto transform}.  

Under the $m$-equivariance assumption~\eqref{equiv}, the Landau-Lifshitz equation \eqref{LL2} is equivalent to the following reduced equation for $\M{v}(r,t)$:
\EQ{  \label{veq}
  \M{v}_t = P^{\M{v}}_a \left[ \p_r^2 + \frac{\p_r}{r} + 
  \frac{m^2}{r^2}R^2 \right] \M{v}. }
Define the operator $\p_{\M{v}}$ on vector-valued functions by 
\EQ{
  \p_{\M{v}} := \p_r - \frac{m}{r}J^{\M{v}} R. }
Since for any vector $\M{b}$, 
$J^{\M{v}} R \M{b} = \M{k}(\M{v}\vdt\M{b})-(\M{k}\vdt\M{v}) \M{b}$, 
we have $J^{\M{v}} R = -v_3$ on the tangent space 
$T_{\M{v}} \S^2 = \M{v}^\perp$. 
For future use, we denote the corresponding operator 
on scalar functions by 
\EQ{
  L_{\M{v}} := \p_r + \frac{m v_3}{r}. }
Then equation~\eqref{veq} can be factored as 
\EQ{  \label{vfac}
  \M{v}_t = - P^{\M{v}}_a D_{\M{v}}^* \p_{\M{v}} \M{v},}
where 
\EQ{
  D := P^{\M{v}} \p P^{\M{v}} } 
will always denote a covariant derivative (which acts on 
$T_{\M{v}} \S^2$-valued functions), and $*$ denotes the adjoint 
in $L^2(\R^2)$. 
Denote the right-most factor in~\eqref{vfac} by 
\EQ{ 
  \M{w} := \p_{\M{v}} \M{v} = \M{v}_r - \frac{m}{r} P^{\M{v}} \M{k}. } 
Then~\eqref{vfac} becomes 
$\M{v}_t = - P^{\M{v}}_a D_{\M{v}}^* \M{w}$, and 
applying $D_{\M{v}}$ to both sides yields
\EQ{ \label{weq}
  \pq D_t \M{w} = - P^{\M{v}}_a D_{\M{v}} D_{\M{v}}^* \M{w}. }
Now we rewrite the equation for $\M{w}$ by choosing an appropriate 
orthonormal frame field on $T_{\M{v}}\S^2$, realized in $\C^3$.
Let $\V{e}=\V{e}(t,r)$ satisfy
\EQ{
  \Re\V{e} \in \M{v}^\perp, \pq |\Re\V{e}|=1, 
  \pq \Im\V{e} = J^{\M{v}} \Re\V{e}.}
Let $S,T$ be real scalar, and let $q,\nu$ be complex scalar, defined by 
\EQ{
 \pt \M{w} = q \cdt \V{e}, \pq P^{\M{v}} \M{k} = \nu\cdt\V{e}, 
 \pq D_t \V{e} = -i S \V{e}, \pq D_r \V{e} = -i T \V{e}.}
Then we have the general curvature relation
\EQ{
 [D_r, D_t]\V{e} = i (T_t-S_r) \V{e} = -i\det\mat{\M{v} & \M{v}_r & \M{v}_t} \V{e}.}
Using the equation~\eqref{vfac} for $\M{v}$, we get
\EQ{ \label{curve}
  S_r - T_t \pt= (\M{w}+\frac{m}{r} P^{\M{v}}\M{k})
  \vdt P^{\M{v}}_{ia} D_{\M{v}}^* \M{w}. }

Now we fix $\V{e}$ by imposing 
\EQ{
  D_r \V{e} = 0, \pq \V{e}(r=\I) = (1,i,0).}
(The unique existence of such $\V{e}$ will be guaranteed by Lemma \ref{coord lem}.)  
Then~\eqref{curve} yields 
\EQ{
  S_r = (q+\frac{m}{r}\nu)\cdt(iaL_{\M{v}}^*q).}
A key observation is that in the Schr\"odinger (non-dissipative) 
case $a=i$, we can pull out the derivative on $q$:
$S_r = (\p_r+\frac{2}{r})(\frac{1}{2}|q|^2 + 
\frac{m}{r} \nu \cdt q)$, and so
\EQ{ \label{SchroS}
  S = Q - \int_r^\I 2Q\frac{dr}{r}, \pq \quad
  Q:=\frac{1}{2}|q|^2 + \frac{m}{r}\nu\cdt q = 
  \frac{1}{2}|\M{w}|^2 + \frac{mw_3}{r}
  \quad\quad (a=i).}
The evolution equation~\eqref{weq} for $w$ yields our
equation for $q$: 
\EQ{ \label{qeq}
  \pt (\p_t - i S) q = -a L_{\M{v}} L_{\M{v}}^* q, 
  \pq \quad S = -\int_r^\I (q+\frac{m}{r}\nu)\cdt(iaL_{\M{v}}^*q)dr.}
This is the basic equation used to establish diffusive ($a_1 > 0$) 
or dispersive ($a_1 = 0$) decay estimates. 
The operator acting on $q$ can be expanded as 
\EQ{
  L_{\M{v}} L_{\M{v}}^* = \p_r^* \p_r + \frac{(m-1)^2}{r^2} + \frac{2m(1-v_3)}{r^2} + \frac{m}{r} w_3.} 


\section{Decomposition and orthogonality}
\label{splitting}
Next we introduce coordinates for the decomposition of the original map
\EQ{
 \M{v} = \hmu + \res{v},} 
or more precisely for the remainder $\res{v}$, and a localized orthogonality condition which determines the decomposition. 
The choice of coordinates is the same as in the previous works \cite{GKT2,GGT,GGKT}, while the decomposition itself is different. 

For each harmonic map profile $\M{h}[\mu]$, $\mu = m \log s + i \al$,
we introduce an orthonormal frame field 
\EQ{
  \V{f} = \V{f}[\mu] := e^{\al R}(-\M{h}^s\times\M{j} + i\M{j}). }
on the tangent space $T_{\hmu} \S^2$, such that the parameter 
derivative of $\hmu$ is given by 
\EQ{
  d\hmu = h_1^s d\mu \cdt\V{f}.} 
We express the difference from the harmonic map in this frame by 
\EQ{ \label{def z}
   z := \res{v}\vdt\V{f}. }
In other words $P^{\hmu}\M{v} = z\cdt \V{f}$, or 
$\res{v}=z\cdt \V{f} +\ga \hmu$, where we denote 
\EQ{ \label{def ga}
  \ga:=\sqrt{1-|z|^2}-1 = -O(|z|^2).}
As explained in the introduction, the orthogonality condition in the previous works
\EQ{
 \ip{z}{h_1^s} = 0}
would not work for $m\le 3$ due to the slow decay of $h_1^s$ for $r\to\I$. 
Hence instead we determine the parameter $\mu$ by imposing 
{\it localized orthogonality} 
\EQ{
  \ip{z}{\fy^s} = 0, \pq \quad \fy^s=\fy(r/s), }
with some smooth localized function $\fy(r)\in C_0^\I((0,\I);\R)$, 
satisfying $\ip{h_1 }{ \fy}=1$. 
The fact that $e^{m \th R} \M{h}[\mu]$ solves~\eqref{cr} means that 
\EQ{
  \p_{\M{h}} \M{h} = 0, }
and so we have
\EQ{ \label{def w}
  \M{w} = \p_{\M{v}} \M{v} = \res{v}_r + \frac{m}{r}(h_3^s\res{v}+\res{v}_3\M{v}) 
  = L^s\res{v}+\frac{m}{r}\res{v}_3 \M{v}. }
Hence   
\EQ{ \label{eq Lz}
 L^sz \pt= L^s\res{v}\vdt\V{f} + \res{v}\vdt \V{f}_r
 \pn= \M{w}\vdt\V{f} - \frac{m}{r}\res{v}_3 z + \frac{m}{r}h_1^s\ga.}
In order to estimate $z$ by $\M{w}$ (or equivalently $q$), we 
introduce a right inverse of the operator 
$L^s = \pd_r + \frac{m}{r} h_3^s$, defined by
\EQ{ 
\label{def R^s}
  \pt R_\fy^s g := h_1^s(r)\int_0^\I \int_{r'}^r 
  h_1^s(r'')^{-1}g(r'') dr'' \rs{\bar\fy}(r')h_1^s(r')r' dr'}
Then we have
\EQ{ \label{LR}
  L^s R_\fy^s g= g, \pq \quad R_\fy^s L^s g = g - h_1^s\ip{g}{\rs\fy}, }
hence $R_\fy^s=(L^s)^{-1}$ on $(\fy^s)^\perp$. Moreover we have 
the following uniform bounds
\begin{lem} \label{bd on R}
For all $p \in [1,\I]$ and $|\th|<m$, we have 
\EQ{
 \pt \|R_\fy^s g\|_{r^{\th}L^\I_p} \lec \|\fy\|_{r^{-\th}L^1_{p'}} \|g\|_{r^{\th+1}L^1_p},}
where the $L^p_q$ norm is defined in~\eqref{Lpq}. 
Moreover, the condition on $\fy$ is optimal in the following sense: if $\fy\ge 0$, then 
$\fy\in r^{-1}L^1_{p'}$ is necessary for $R_\fy^s$ to be bounded $r^{\th-1}L^\I_p\to\D'(0,\I)$. 
\end{lem} 
We give a proof in Section~\ref{pf lin ests}. 
Note that the above bounds are scaling invariant: denoting 
$D_s f := f(r/s)$, we have
\EQ{
  \pt R_\fy^s = s D_s R_\fy^1 D_s^{-1},
  \pq L^s = s^{-1}D_s L^1 D_s^{-1}.}
We can combine the estimates of the Lemma with the embedding 
\EQ{
 r^{\th_1} L^{p_1}_{q_1} \subset r^{\th_2} L^{p_2}_{q_2} \iff \frac{2}{p_1}-\th_1 = \frac{2}{p_2}-\th_2,\ p_1\ge p_2,\ q_1\le q_2.} 
 
The above lemma is used as follows. 
First note that the orthogonality $\ip{\fy^s}{z}=0$ implies that $z=R_\fy^s L^s z$ because of~\eqref{LR}. 
For the energy norm, we choose $\th=0$ and $p=2$ in Lemma~\ref{bd on R}. Then   
\EQ{ \label{en bd 1 z}
  \pt\|z/r\|_{L^2_x} \lec \| z \|_{L^\I_2} 
  \lec \|\fy\|_{L^1_2} \|L^sz\|_{r L^1_2} \lec \|L^sz\|_{L^2_x}.}
Since $|L^s-\p_r|\lec 1/r$, we further obtain
\EQ{ \label{R bd 2 X}
  R_\fy^s : r^\th L^2 \to r^\th X, \pq (|\th|<m),}
where the space $X$ is defined by the norm 
\EQ{
  \|z\|_X := \|z/r\|_{L^2_x} + \|z_r\|_{L^2_x}.} 
The Sobolev embedding $X\subset L^\I$ is trivial by Schwarz: 
\EQ{ \label{Sob X}
 \|z\|_{L^\I_x}^2 \le \|z/r\|_{L^2_x} \|z_r\|_{L^2_x}.} 
Hence we get by using \eqref{eq Lz}, 
\EQ{ \label{en bd z}
 \|z\|_X \lec \|L^sz\|_{L^2} \lec \|q\|_{L^2} + \|z\|_{L^\I}\|z\|_X.}

For $L^2_t$ estimates of $z$, we use Lemma~\ref{bd on R} with $\th=1$ and $p=\I$. 
Then we have 
\EQ{ \label{str bd z}
  \pt\|z/r\|_{L^\I_p} \lec \|\fy\|_{r^{-1}L^1_1} 
  \|L^sz\|_{r^2 L^1_p} \lec \|L^sz\|_{L^\I_p},}
for any $p\in[1,\I]$, and so by using \eqref{eq Lz},
\EQ{ \label{Str bd z}
 \|z/r\|_{L^2_t L^\I_p} \lec \|q\|_{L^2_t L^\I_p} + \|z\|_{L^\I_{t,x}}\|z\|_{L^2_t L^\I_p}.}

If we were to use $h_1$ instead of $\fy$, then we 
would need $m>3$ for the Strichartz-type bound \eqref{Str bd z}, and 
$m>2$ for the energy bound \eqref{en bd 1 z}, by the last statement 
of the lemma.


\section{Coordinate change}
\label{coord}
Before beginning the estimates for the evolution, we establish in this 
section the bi-Lipschitz correspondence between the different coordinate systems: $\M{v}$ and $(\mu,q)$, including unique existence of the decomposition. 
It is valid for any map in our class $\Si_m$ with energy close to the ground states.

For that purpose, we need to translate between the different frames $\V{e}$ 
and $\V{f}$.  At each point $(t,r)$, we define 
$\Sc{M}=\V{f}\otimes\V{e}\in GL_\R(\C)$, a real-linear map 
$\C\to\C$, by 
\EQ{
  \Sc{M} z := \V{f}\vdt(\V{e}\cdt z) . } 
Its transpose $\Tr{\Sc{M}}=\V{e}\otimes\V{f}$, defined by 
$\Tr{\Sc{M}}z=\V{e}\vdt(\V{f}\cdt z)$, is the adjoint in the sense 
that $(\Sc{M}z)\cdt w = z \cdt(\Tr{\Sc{M}}w)$. 
For any $\V{b},\V{c},\V{d}\in\C^3$ we have 
\EQ{
 \V{b} \vdt (\V{c} \cdt \V{d}) = (\Re\V{b}\vdt\V{c})\cdt\V{d} + (\Im\V{b}\vdt\V{c})\cdt\V{d}.}
Since $\V{f}(\I)=e^{-i\al}\V{e}(\I)$, and $\V{f}\perp\hmu$, we have  
\EQ{ \label{eq M} 
  \Sc{M}(\I)=e^{-i\al}, \pq 
  \Sc{M}_r = \V{f}\otimes\V{e}_r + \V{f}_r\otimes\V{e} \pn = 
  - \V{f}\otimes \res{v} (\V{e}\vdt\M{v}_r)  
  - \frac{m}{r}h_1^s \res{v} \otimes \V{e}. } 
Then $\V{e}$ can be recovered from $\Sc{M}$ by
\EQ{ \label{e from M}
  \V{e} = P^{\hmu}\V{e} + (\hmu\vdt\V{e})\hmu 
  = \Tr{\Sc{M}}\V{f} - (1+\ga)^{-1}(\Tr{\Sc{M}}z)\hmu,}
provided that $|\ga|<1$. 
We further introduce some spaces with (pseudo-)norms.  
\EQ{
  \pt \E_m(\M{v})= \E(e^{m\th R}\M{v}(r)),  \pq 
  |\mu|_C = \min(|\Re\mu|,1)+dist(\Im\mu,2\pi\Z), 
  \pr \|z\|_X = \|z/r\|_{L^2_x}+\|z_r\|_{L^2_x}, \pq 
  \|\Sc{M}\|_Y = \|\Sc{M}_r\|_{L^1(dr)} + \|\Sc{M}\|_{L^\I_r} \pr 
  \Si_m(\de) = \{\M{v}(r):[0,\I]\to \S^2 \mid \M{v}(0)=-\M{k},\ 
  \M{v}(\I)=\M{k},\ \E_m(\M{v})\le 4m\pi + \de^2\},
  \pr L^2(\de)=\{q(r):[0,\I)\to\C\mid \|q\|_{L^2_x}\le \sqrt{2}\de\}, \pq 
  C = \C/2\pi i\Z.}
The metric on $C$ is defined such that 
\EQ{
  \|\M{h}[\mu^1] - \M{h}[\mu^2]\|_X \sim 
  \|\M{h}[\mu^1]-\M{h}[\mu^2]\|_{L^\I} \sim |\mu^1-\mu^2|_C.}
\begin{lem} \label{coord lem}
Let $m\in\N$ and $\fy\in C_0^1(0,\I)$ satisfy $\ip{\fy}{h_1}=1$. 
Then there exists $\de>0$ such that the system of equations  
\EQ{ \label{coord eq}
  \pt \M{v} = \hmu + \res{v}, \pq z = \M{v} \vdt\V{f}[\mu], \pq 
 \ip{z }{ \fy^s}=0, \pq \ga=\sqrt{1-|z|^2}-1,
 \pr q = \V{e}\cdt(\M{v}_r - \frac{m}{r}P^{\M{v}}\M{k}), 
 \pq D_r \V{e}=0, \pq \V{e}(\I)=(1,i,0),}
defines a bijection from $\M{v}\in \Si_m(\de)$ to $(\mu,q)\in C\times L^2(\de)$, which is unique under the condition $\|z\|_{L^\I_x}\lec\de$. $\res{v}$, $z$ and $\V{e}$ are also uniquely determined. 
Moreover, if $(\M{v}^j,\dots,\V{e}^j)$ with $j=1,2$ are such tuples given in this way, then we have 
\EQ{
 \pt \|\res{v}^1-\res{v}^2\|_X  
   + \|z^1-z^2\|_X + \|\V{e}^1-\V{e}^2\|_{L^\I} +\|\Sc{M}^1-\Sc{M}^2\|_Y
 \pr\lec \|\M{v}^1-\M{v}^2\|_{X} \pn\sim |\mu^1-\mu^2|_C+\|q^1-q^2\|_{L^2},}
where $\Sc{M}^j:=\V{f}^j\otimes\V{e}^j$.
\end{lem}
In particular, we have pointwise smallness 
\EQ{ \label{Linf bd}
 \|\res{v}\|_{L^\I_x} \sim \|z\|_{L^\I_x} \lec \de \ll 1,}
so that we can neglect higher order terms in $z$ or $\res{v}$.
\begin{proof} 
We always assume \eqref{def w}, \eqref{def z} and \eqref{def ga}, 
which define the maps 
\EQ{
 \M{v}\mapsto\M{w}=q\cdt\V{e}, \pq (\M{v},\mu)\mapsto\res{v}\leftrightarrow z\mapsto \ga, \pq (\res{v},\mu)\mapsto \M{v},}
with the Lipschitz continuity
\EQ{ \label{w z from v}
 \pt \|\M{w}^1 - \M{w}^2\|_{L^2} \lec \|\M{v}^1-\M{v}^2\|_X,
 \pq \|\ga^1-\ga^2\|_X \lec \|z^1-z^2\|_X \sim \|\res{v}^1-\res{v}^2\|_X,
 \pr \bigl|\|\M{v}^1-\M{v}^2\|_X - \|\res{v}^1-\res{v}^2\|_X\bigr| \lec |\mu^1-\mu^2|_C.}
The energy can be written as 
\EQ{
  2 \E_m(\M{v})  \pt = 
  \|\M{v}_r\|_{L^2_x}^2 + \left\|\frac{m}{r} R \M{v} \right\|_{L^2_x}^2 
  = \|\M{v}_r\|_{L^2_x}^2 + \left\|\frac{m}{r} P^{\M{v}}\M{k}\right\|_{L^2_x}^2 
  \pr= \|\M{w}\|_{L^2_x}^2 + 2m\ip{\M{v}_r}{P^{\M{v}}\M{k}/r}
  = \|\M{w}\|_{L^2_x}^2 + 4\pi[v_3(\I)-v_3(0)].}
Since $\|P^{\M{k}}{\M{v}}\|_X \lec \E_m({\M{v}})^{1/2}$, $X\subset L^\I_x$ and $|{\M{v}}|=1$, the boundary conditions $v_3(0)=-1$ and $v_3(\I)=1$ make sense in the energy norm. 

Next we consider a point orthogonality. 
Let $\M{v}\in\Si_m(\de)$. 
Since $v_3(0)<0<v_3(\I)$ and $v_3(r)$ is continuous, we have
$\M{v}(s_0)=e^{i\al_0R}\M{h}(1)$ for some $\mu_0=m\log s_0+i\al_0$, so
that $\M{v}=\M{h}[\mu_0]+\res{v}$ is a decomposition satisfying
$\ip{z}{\fy^{s_0}} = 0$ if $\fy(r)=\de(r-1)$. 
In this case $\M{v}$ is recovered from $(\M{w},\mu_0)$ by solving the ODE: 
\EQ{
 \pt L^sz = \M{w}\vdt\V{f}[\mu_0] - \frac{m}{r}\res{v}_3 z + \frac{m}{r}h_1^{s_0}\ga, \pq z(s_0)=0,}
or the equivalent integral equation 
\EQ{
 z = R_{\de(r-1)}^{s_0}\left[\M{w}\vdt\V{f}[\mu_0] - \frac{m}{r}\res{v}_3 z + \frac{m}{r}h_1^{s_0}\ga\right].}
The uniform bound on $R_{\de(r-1)}^s$ can be localized onto any
interval $I\ni s_0$, because $z$ is the solution of the above 
initial value problem. Hence we get, in the same way as in \eqref{en bd 1 z}, 
\EQ{
 \|z\|_{rL^2_x \cap L^\I_x(I)} \lec \|q\|_{L^2_x(I)} + \|z\|_{L^\I_x(I)}\|z\|_{rL^2_x(I)}.}
Since $z(s_0)=0$ and $\|q\|_{L^2_x} \le \de\ll 1$, we get by continuity in $r$ for $I\to(0,\I)$,  
\EQ{ \label{z bd very loc}
  \|z\|_{X \cap L^\I_x} \lec \|q\|_{L^2_x} \lec  \de.}
Thus every $\M{v}\in\Si_m(\de)$ is close at least to some $\M{h}[\mu_0]$, and   
we have $\M{v}^1-\M{v}^2\in X$ by \eqref{w z from v}. 
$\Si_m(\de)$ is a complete metric space with this distance. 

Now we take any $\fy\in C_0^1(0,\I)$ satisfying $\ip{\fy}{h_1}=1$,  
and look for $\mu$ around $\mu_0$ solving the orthogonality   
\EQ{
 F(\mu) := \ip{\M{v}\vdt\V{f}[\mu]}{\rs\fy} = \ip{\res{v}\vdt\V{f}[\mu]}{\rs\fy} = \ip{z}{\rs\fy} = 0.}
Its derivative in $\mu$ is given by
\EQ{
 dF \pt= -\ip{\M{v}\vdt h_1^s\hmu }{\rs\fy}d\mu -i\ip{\M{v}\vdt h_3^s\V{f}[\mu]}{\rs\fy}d\al - \ip{\M{v}\vdt\V{f}[\mu]}{(r\p_r+2)\rs\fy}\frac{ds}{s}
 \pr = -d\mu -\ip{\res{v}\vdt\hmu}{h_1^s\rs\fy}d\mu - \ip{\res{v}\vdt\V{f}[\mu]}{(r\p_r+2)\rs\fy ds/s+ih_3^s \rs\fy d\al}
 \pr= -d\mu + O(\de|d\mu|).}
In particular we have 
\EQ{
 |F(\mu_0)|\lec \de, \pq \frac{\p F}{\p\mu}(\mu_0) = -I + O(\de).}
In addition, both $F(\mu)$ and $\p_\mu F$ are Lipschitz in $\M{v}$. 
Therefore by the implicit mapping theorem, if $\de>0$ is small enough, there exists a unique $\mu\in\C$ for each $v$ such that $F(\mu)=0$ and $|\mu-\mu_0|\lec\de$, and $\M{v}\mapsto \mu$ is Lipschitz. Then
\EQ{
 \|z\|_{L^\I_x} \lec \|\M{v}-\M{h}[\mu_0]\|_{L^\I_x}+|\mu_0-\mu| \lec \de \ll 1,}
and so by the same argument as for \eqref{z bd very loc}, we get $\|z\|_{X}\lec\de$, and in addition, 
\EQ{
 \|z^1-z^2\|_X \lec |\mu^1-\mu^2|_C + \|\M{w}^1-\M{w}^2\|_{L^2}.}
If we have two such $\mu=\mu_1,\mu_2$ with $\|z^j\|_{L^\I}\lec\de$, then 
\EQ{
 |\mu_1-\mu_2|_C \sim \|\M{h}[\mu_1]-\M{h}[\mu_2]\|_{L^\I_x} \pt\lec \|\M{v}-\M{h}[\mu_1]\|_{L^\I_x} + \|\M{v}-\M{h}[\mu_2]\|_{L^\I_x} \lec \de,}
and so the implicit mapping theorem implies that $\mu_1=\mu_2$. 
Thus we get a bijection $\M{v}\mapsto(\mu,\M{w})$ with the Lipschitz continuity
\EQ{
 \|\M{v}^1-\M{v}^2\|_X \sim |\mu^1-\mu^2|_C + \|\M{w}^1-\M{w}^2\|_{L^2_x}.}  

For the frame field $\V{e}$, we consider the matrix $\Sc{M}=\V{f}\otimes\V{e}$, together with the equivalent set of equations \eqref{eq M} and \eqref{e from M}.
Integrating \eqref{eq M} from $r=\I$, we get 
\EQ{ \label{est M diff}
 \pt \|\Sc{M}-e^{i\al}\|_Y \lec \|\res{v}/r\|_{L^2_x}\|\M{v}_r\|_{L^2_x} + \|\res{v}/r\|_{L^2_x}\|h_1^s/r\|_{L^2_x} \lec \de,
 \pr \|\Sc{M}^1-\Sc{M}^2\|_Y \lec |\mu^1-\mu^2|_C + \de\|\M{v}^1-\M{v}^2\|_X + \de\|\V{e}^1-\V{e}^2\|_{L^\I} + \|\res{v}^1-\res{v}^2\|_X,}
while \eqref{e from M} provides 
\EQ{ \label{est e diff}
 \|\V{e}^1-\V{e}^2\|_{L^\I_x} \pt\lec \|\Sc{M}^1-\Sc{M}^2\|_{L^\I_x} + |\mu^1-\mu^2|_C + \|z^1-z^2\|_{L^\I_x}.}
Hence for fixed $\M{v}\in\Si_m(\de)$ (and $\mu$), we can get $(\Sc{M},\V{e})\in Y\times L^\I$ by the contraction mapping principle for the system of \eqref{eq M} and \eqref{e from M}. Moreover we get 
\EQ{
 \pt \|\Sc{M}^1-\Sc{M}^2\|_Y + \|\V{e}^1-\V{e}^2\|_{L^\I_x} \lec \|\M{v}^1-\M{v}^2\|_X.}

If $(\mu,q)\in C\times L^2(\de)$ is given, we consider the system of equations \eqref{eq M}, \eqref{e from M} and 
\EQ{
 z = R_\fy^s\left[\Sc{M} q - \frac{m}{r}\res{v}_3z + \frac{m}{r}h_1^s\ga\right],}
which is equivalent to the $q$ equation in \eqref{coord eq} under the orthogonality $\ip{z}{\rs{\fy}}=0$. 
The last equation provides, through the uniform bound on $R_\fy^s$, 
\EQ{
 \|z^1-z^2\|_X \lec |\mu^1-\mu^2|_C + \|q^1-q^2\|_{L^2_x} + \de\|\Sc{M}^1-\Sc{M}^2\|_Y + \de\|z^1-z^2\|_X.}
Combining this with \eqref{est M diff} and \eqref{est e diff}, we get $(z,\Sc{M},\V{e})$ for any fixed $(\mu,q)$ by the contraction mapping, and moreover they satisfy
\EQ{
 \|z^1-z^2\|_X + \|\Sc{M}^1-\Sc{M}^2\|_Y + \|\V{e}^1-\V{e}^2\|_{L^\I} \lec |\mu^1-\mu^2|_C + \|q^1-q^2\|_{L^2_x}.}

\end{proof}

\section{Decay estimates for the remainder} \label{decay}
In this section we derive dissipative or dispersive space-time estimates of the remainder $\res{v}$ in terms of $z$, from the equation \eqref{qeq} for $q$. First by the smallness of $z$, we obtain from \eqref{en bd z} and \eqref{Str bd z}, 
\EQ{ \label{z bd q}
 \pt \|z(t)\|_X \lec \|q(t)\|_{L^2} \lec \de, 
 \pq \|z/r\|_{L^\I_p} \lec \|q\|_{L^\I_p},}
for all $p\in[1,\I]$. 
Next we estimate the factor $S$, by using  
\EQ{
  \| r \int_r^\I f g dr \|_{L^\I_1} \pt\lec \sum_{j\in\Z} \sum_{k\ge j} 2^{j-k} \|fg\|_{L^1(r\sim 2^k)}
  \pn\sim \|fg\|_{L^1} \le \| f \|_{L^2} \| g \|_{L^2}. }
Then from the expression in~\eqref{qeq} for $S$, we have
\EQ{ \label{SEstDiss} 
  \|S(t)\|_{L^2_1} \lec \| S(t) \|_{r^{-1} L^\I_1} \lec
  ( \| q \|_{L^2_x} + \| z \|_{rL^2_x} + 1)
  \| L_{\M{v}}^* q \|_{L^2_x} \lec \|L_{\M{v}}^* q\|_{L^2_x}.}
In the dispersive case $a_1 = 0$, we avoid the derivative by using expression~\eqref{SchroS}
\EQ{ \label{SEstSchro}
  \| S(t) \|_{L^2_1} \lec \| S(t) \|_{r^{-1} L^\I_1} \lec
  ( \|q \|_{L^2_x} + \| z \|_{rL^2_x} + 1 )
  \| q \|_{L^\I_2} \lec \|q\|_{L^\I_2} \quad (a_1 = 0). }
For the time decay estimates, we treat the dissipative and the 
dispersive cases separately.

\subsection{Dissipative $L^2_t$ estimate}
Here we assume $a_1 > 0$.  By the equation~\eqref{qeq} of $q$, we have
\EQ{
 \p_t\|q\|_{L^2}^2 = -2 a_1 \|L_{\M{v}}^*q\|_{L^2}^2,}
hence
\EQ{
 \|q\|_{L^\I_t L^2_x} + \|L_{\M{v}}^*q\|_{L^2_tL^2_x} \lec \|q(0)\|_{L^2_x} \sim \de.}
Since $R^{s*}_\fy L^{s*}=I$ and $R^{s*}_\fy :L^2\to rL^2$ by Lemma \ref{bd on R} and duality, we have 
\EQ{
 \|q\|_{rL^2_x} \lec \|L^{s*}q\|_{L^2_x} \lec \|L_{\M{v}}^*q\|_{L^2_x} + \|\res{v}\|_{L^\I}\|q\|_{rL^2_x}.}
Since the last term can be absorbed by \eqref{Linf bd} smallness of $\res{v}$, we get
\EQ{
 \|q\|_X \lec \|q/r\|_{L^2_x} + \|L^{s*}q\|_{L^2_x} \lec \|L_{\M{v}}^*q\|_{L^2_x}.}
So by using the bound \eqref{R bd 2 X} on $R_\fy^s$, we obtain 
\EQ{ \label{q L^2 diss}
 \|z\|_{rL^2_t X} \lec \|q\|_{L^2_t X} \lec \|L_{\M{v}}^*q\|_{L^2_{t,x}} \lec \|q(0)\|_{L^2_x} \sim\de,}
and also from \eqref{SEstDiss}, 
\EQ{ \label{S L^2 diss}
 \|S\|_{L^2_t L^2_x} \lec \de.}

\subsection{Dissipative decay}
Next we show the convergence $q\to 0$ as $t\to\I$, by comparing it 
with the free evolution. For $T>0$, let  
\EQ{
 q^T := q-e^{(t-T)a\De_2^{(m-1)}}q(T).}
Then we have 
\EQ{
  q^T_t - a \De_2^{(m-1)} q^T = (iS - aV)q, \pq q^T(T)=0,}
where the potential $V(t,x)$ is given by
\EQ{
  V = \frac{2m(1-v_3)}{r^2} + \frac{m}{r} w_3.}
Multiplying the equation with $q^T$, we get the energy identity 
\EQ{
  \pt \frac{1}{2} \|q^T\|_{L^2_x}^2 + \int_T^t a_1(\|q^T_r\|_{L^2_x}^2+\|\frac{m-1}{r}q^T\|_{L^2_x}^2) dt 
  =  \Re \int_T^t \ip{-aVq+iSq }{ q^T} dt,}
and hence by Schwarz, and using estimate~\eqref{SEstDiss} to 
put $S \in L^2_t L^2_x$, 
\EQ{
 \|q^T\|_{L^\I_{t>T} L^2_x \cap L^2_{t>T} X} 
   \pt\lec \|q/r\|_{L^2_{t>T}L^2_x} + \|Sq\|_{L^2_{t>T}L^1_x}
   + \| q \|_{L^4_{t>T} L^4_x}^2 
   \pr\lec \|q/r\|_{L^2_{t>T}L^2_x} 
   + \| q \|_{L^\I_t L^2_x} \| q \|_{L^2_{t>T} X} \to 0\ 
   \quad (T\to\I),}
Hence $\|q(t)\|_{L^2_x}$ can not converge to a positive number, since $e^{(t-T)a\De_2^{(m-1)}}q(T)\to 0$ as $t\to\I$ for all $T>0$. Thus we obtain 
\EQ{
  \|z(t)\|_X \lec \|q(t)\|_{L^2_x} \to 0 \pq(t\to\I).}

\subsection{Dispersive $L^2_t$ estimate}
Next we consider the case $a_1=0$ (and $a_2\not=0$).
We set (with no loss of generality) $a=i$.
Since the energy identity provides only $L^2_x$ bound on $q$, 
we have to work with the Strichartz estimate in a perturbative way. 
Denoting $H^s := L^s {L^s}^*$, the equation of $q$ is given by
\EQ{
 \pt q_t +i H^{s(0)}q = N_1 + N_2,}
where 
\EQ{ \label{def N1N2}
 N_1 := -2am\frac{h_3^{s(0)}-h_3^{s(t)}}{r^2} q, \pq N_2:= iSq -2am\frac{\res{v}_3}{r^2} q - am \frac{ w_3}{r}q,}
and $S$ is given by \eqref{SchroS}. 
We have
\EQ{
 |N_1| \lec |h_3(s(t)/s(0))||q|/r^2,}
and so
\EQ{
 \|N_1\|_{L^2_t L^1_2} \lec \|h_3(s(t)/s(0))\|_{L^\I_t} \|q\|_{L^2_t L^\I_2}.}
Using \eqref{SEstSchro}, we have 
\EQ{
 \|Sq\|_{L^1_t L^2_x} \le \|S\|_{L^2_t L^2_x}\|q\|_{L^2_t L^\I_x} \lec \|q\|_{L^2_t L^\I_2}^2.}
The other terms in $N_2$ are bounded in $L^1_tL^2_x$ by 
\EQ{
 \|q\|_{L^2_t L^\I_2}^2 + \|z/r\|_{L^2_t L^\I_2}\|q\|_{L^2_t L^\I_2} \lec \|q\|_{L^2_t L^\I_2}^2.}
Now we need the endpoint Strichartz estimate for $H^s$ with fixed scaling $s$:
\begin{lem} \label{Str H}
Let $H^s=L^s {L^s}^*=-\De^{(m-1)}_2+2mr^{-2}(1-h_3^s)$ and $m>1$. Then we have 
\EQ{
 \pt \|e^{-iH^st}\fy\|_{L^\I_t L^2_x \cap L^2_t L^\I_2} \lec \|\fy\|_{L^2_x}
 \pr \|\int_{-\I}^t e^{-iH^s(t-t')}f(t')dt'\|_{L^\I_t L^2_x \cap L^2_t L^\I_2} \lec \|f\|_{L^1_t L^2_x + L^2_t L^1_2},}
uniformly for any fixed $s>0$. 
\end{lem}
This Lemma will be proved in Section~\ref{end}.
Hence if $|\log(s(t)/s(0))|\ll 1$ for all $t$, then we have 
\EQ{
 \|q\|_{L^\I_t L^2_x \cap L^2_t L^\I_2} \lec \|q(0)\|_{L^2} \sim \de,}
and also from \eqref{SEstSchro}
\EQ{ \label{S L^2 disp}
 \|S\|_{L^2_t L^2_x} \lec \de. }

\subsection{Dispersive decay}
Next we prove the following asymptotics of scattering type for $q$ and $z$:
\EQ{
   e^{-it\De_2^{(m-1)}}q(t)\to\exists q_+\IN{L^2_x}, \pq 
   z\to 0 \IN{L^\I_x} \pq (t\to\I). }
For the scattering of $q$, we further expand the equation 
\EQ{ \label{qExpand}
  q_t -i \De_2^{(m-1)} q = N_0 + N_2,}
where $N_2$ is as in \eqref{def N1N2}, and 
\EQ{ \label{def N0}
  N_0 :=  -2am\frac{1-h_3^{s}}{r^2}q}
Then the global Strichartz bound implies that 
\EQ{
  \|N_0\|_{L^2_t L^1_2(T,\I)} \to 0,\pq 
  \|N_2\|_{L^1_t L^2_x(T,\I)}\to 0 }
as $T\to\I$. By Strichartz (for $\De_2^{(m-1)}$) once again, we get 
the scattering of $q$. 

For the vanishing of $z$, we use the inversion formula 
\EQ{
 z = R_\fy^s g, \pq g= \Sc{M}q  + r^{-1}m(h_1^s\ga - \res{v_3}z). }
Since $R_\fy^s$ is bounded $L^2_x\to L^\I$, the latter two terms contribute at most with $\|z\|_{L^\I_x}\|z/r\|_{L^2_x}\ll \|z\|_{L^\I_x}$, hence we may drop them. 
Also we may replace $q$ by its asymptotic free solution $q^\I:=e^{it\De_2^{(m-1)}}q_+$. 
Moreover we may approximate $q_+$ by nicer functions. 
Hence we assume that $\hat q:=\F_{m-1}q_+\in C_0^\I(0,\I)$. 
Then we may further replace the free solution with the stationary phase part: 
\EQ{
 q^\I(t,r) \pt= C_mt^{-1}e^{ir^2/(4t)} \int_0^\I J_{m-1}(r\r/(2t))e^{i\r^2/(4t)}q_+(\r)\r d\r
 \pr= C_mt^{-1}e^{ir^2/(4t)}\hat q(r/(2t)) + \Sc{R},}
where the error is bounded by Plancherel
\EQ{
 \|\Sc{R}\|_{L^2_x} \sim \|(1-e^{ir^2/(4t)})q_+\|_{L^2_x} \lec t^{-1}\|r^2q_+\|_{L^2_x} \to 0.}
Now that spatially local vanishing is clear 
(eg. it follows from $\| R^s_\fy M q(t) \|_{r L^\infty}
\lec \| q^\I(t) \|_{L^\I} \to 0$),
we may extract the leading term of $R_\fy^s$ for large $x$. 
We assume that $s(t)\in L^\I_t$ and $\supp\fy^s\subset(0,b)$ for a fixed $b\in(0,\I)$. 
Then for $r>b$ we have 
\EQ{
 (R^s_\fy g)(r) \pt= o(1) + \int_b^r
  \frac{h_1^s(r)}{h_1^s(r'')}g(r'')dr'' 
 \pr= o(1) + \int_b^r (r'/r)^m
  g(r')dr' \mbox{ as } r \to \I.}
Thus we are reduced to showing that
\EQ{
 G \chi:= \int_b^r (\r/r)^m \Sc{M}(t,\r) t^{-1}e^{i\r^2/(4t)}\chi(\r/t)  d\r \to 0 \IN{L^\I_r}}
for any $\chi\in C_0^\I(0,\I)$. By partial integration on $(\r/t)e^{i\r^2/(4t)}$, we have 
\begin{multline}
 r^{m}G \chi = (i/2) [\r^{m-1}\Sc{M}(\r) e^{i\r^2/(4t)}\chi(\r/t)]_b^r\\
  - \int_b^r \bigl[(m-1)\Sc{M}(\r)\chi(\r/t)/\r + \Sc{M}(\r)\chi'(\r/t)/t \\
  + \Sc{M}_r(\r)\chi(\r/t)\bigr]\r^{m-1}e^{i\r^2/(4t)} d\r,
\end{multline}
The right hand side is bounded by $r^m/t$, using $|\chi(\r/t)|\lec\r/t$ for the first, second and fourth terms, $|\chi'(\r/t)|\lec 1$ for the third, and $\Sc{M}_r \in L^\I_t L^1(dr)$ for the fourth term. 
Thus we obtain $\|z(t)\|_{L^\I_x}\to 0$. 

\section{Parameter evolution} \label{parameter}
It remains to control the asymptotic behavior of the parameter 
$\mu(t)$ of the harmonic map part of the solution. 
Its evolution is determined by differentiating the localized orthogonality condition 
\EQ{
  0 \pt= \p_t\ip{z}{\fy^s} = \ip{\res{v}_t\vdt\V{f}}{\fy^s} + 
  \ip{\res{v}\vdt\V{f}_t}{\fy^s} + \ip{z }{ \p_t\fy^s},}
and each term on the right is expanded by using 
\EQ{
 \pt \res{v}_t = \M{v}_t - \hmu_t = -(aL_{\M{v}}^*q) \cdt\V{e} - h_1^s\dot\mu\cdt\V{f},
 \pr \V{f}_t = -ih_3^s\dot\al \V{f} - h_1^s\dot\mu h[\mu],
 \pq \p_t\fy^s = -\frac{\dot s}{s}r\p_r\fy^s.}
Plugging this into the above and then dividing it by $s^2$, we get
\EQ{ \label{eq mu}
  \dot\mu = -\ip{\Sc{M} a L_{\M{v}}^*q }{ \rs\fy} -  
  \ip{h_1^s\dot\mu\ga }{\rs\fy} - (z \; | \; (\frac
  {\dot\mu_1}{m}r\p_r - i\dot\mu_2 h_3^s ) \rs\fy). }
The last two terms are bounded by
\EQ{
 |\dot\mu| \|z\|_{L^\I} (\|\fy\|_{L^1}+\|r\p_r\fy\|_{L^1}),}
and so absorbed by the left hand side since $\|z\|_{L^\I}\lec\de\ll 1$. 

Since $|\nu|=|P^v \M{k}|\lec h_1^s+|\res{v}|$ and hence
\EQ{
  |\M{v}_r| \lec |q| + |z|/r + h_1^s/r, }
we get from \eqref{eq M}, 
\EQ{
 |\Sc{M}_r| \lec |qz| + |z|^2/r + |zh_1^s|/r.}

The leading (first in the r.h.s) term in \eqref{eq mu} can be estimated,
using $[\Sc{M}a, L_{\M{v}}^*] = \Sc{M}_r a$, as follows  
\EQ{
  |\ip{\Sc{M} a L_{\M{v}}^*q }{\rs\fy}| \pt\lec s^{-1} 
 (\|\Sc{M}_r\|_{L^2_x}+\|\Sc{M}\|_{L^\I_x}) \|q\|_{rL^2_x}
 \pr\lec s^{-1} (\|q\|_{L^2_x}+\|z/r\|_{L^2_x}+1)\|q\|_{rL^2_x}.}
Hence using that 
$\|z\|_{L^\I_x}\lec\|z\|_X\lec\|q\|_{L^2}\lec\de\ll 1$, we get
\EQ{ \label{muL2}
  \|s\dot\mu\|_{L^2_t} \lec \|q/r\|_{L^2_{t,x}} \lec \|q\|_{L^2_t L^\I_2}. }


\section{Partial integration for the parameter dynamics}
\label{normal form}
Now we want to integrate in $t$ the right hand side of \eqref{eq mu}, 
which is not bounded in $L^1_t$. 
The key idea is to employ the $q$ equation \eqref{qeq}, by 
identifying a factor of $L_{\M{v}}L_{\M{v}}^*q$, through a partial 
integration in space. 

For the spatial integration, we first freeze the phase factor $\Sc{M}$. 
Since $\hmu=\M{v}=-\M{k}$ at $r=0$, we have $\Sc{M}(t,0)=e^{i\ti\al}$, i.e. $\V{f}(t,0)=e^{i\ti\al}\V{e}(t,0)$ for some real $\ti\al(t)$. 
Then $D_t \V{f}(t,0) = i\ti\al'(t) \V{f}(t,0) - i S(t,0)\V{f}(t,0)$, and so 
\EQ{
 \ti\al'(t) = S(t,0)  + \al'(t).}
We decompose  
\EQ{
 \Sc{M} = e^{i\ti\al} + \res{\Sc{M}},}
and rewrite the leading term of \eqref{eq mu} as follows. 
Let $c=\|h_1\|_{L^2}^{-2}$. Since $L_{\M{v}}=L^s+m\res{v}_3/r$ and $L^sh_1^s=0$, we have 
\EQ{ \label{eqmu lead}
  \ip{\Sc{M} aL_{\M{v}}^*q }{ \rs\fy} \pt= a e^{i\ti\al}
  \ip{L_{\M{v}}^*q }{\rs\fy} + \ip{\res{\Sc{M}} a L_{\M{v}}^*q }{\rs\fy}
  \pr= ae^{i\ti\al}\bigl[\ip{L_{\M{v}}^*q }{ \rs{(\fy - c h_1)}} +
  \ip{mq\res{v}_3/r }{c\rs{h_1}}\bigr] \pr \quad + 
  \ip{ \res{\Sc{M}} a q }{ L_{\M{v}} \rs\fy }
  + \ip{ \res{\Sc{M}}_r a q }{ \rs\fy }.}
The second term is bounded by $\|q\res{v}r^{-3}\|_{L^1_x}\lec
\|q/r\|_{L^2_x}\|z/r^2\|_{L^2_x}$, and the last two terms
are bounded by 
\EQ{
 \|q/r\|_{L^2_x} (\|\Sc{M}_r/r\|_{L^2_x} + \|\res{\Sc{M}}/r\|_{L^\I_x}).}
where the last factor is further bounded by using that 
$\res{\Sc{M}}=0$ at $r=0$
\EQ{ 
 \|\res{\Sc{M}}/r\|_{L^\I_x} \lec \|\Sc{M}_r/r\|_{L^2_x} \lec \|q/r\|_{L^2_x} + \|z/r^2\|_{L^2_x} \lec \|q\|_{L^\I_2}.}
We further rewrite the remaining (main) term. By the definition of $c$, we have 
\EQ{
  \ip{\fy - c h_1 }{ h_1} = 1 - c \|h_1\|_{L^2}^2 = 0,}
and so we have 
\EQ{
  \fy^s - ch_1^s = L^{s*}R_\fy^{s*}(\fy^s-ch_1^s),}
where the operator $R_\fy^s$ was defined in \eqref{def R^s}. Let 
\EQ{
  \psi := R_\fy^*(\fy - c h_1) = -\frac{c}{m-1}r^{1-m} + O(r^{1-3m}) 
  \pq(r\to\I), }
where the asymptotic form easily follows from the fact that  
\EQ{
  \psi(r) = -c(h_1(r)r)^{-1}\int_r^\I h_1(r')^2r'dr' \pq (r\gg 1).}
Then we have, by using equation~\eqref{qeq} for $q$, 
\EQ{
 \pn(-aL_{\M{v}}^*q|\rs{(\fy - c h_1)}) \pt= \ip{-aL^s L_{\M{v}}^*q}{\psi^s/s}
 \pr= \ip{q_t - i Sq}{\psi^s/s} + \ip{amq}{L_{\M{v}} \res{v}_3 r^{-1}
   \psi^s/s},}
and, using~\eqref{def w}, the last term is bounded by
\EQ{
 \|q(|q|+|\res{v}/r|)r^{-2}\|_{L^1_x} \lec (\|q/r\|_{L^2_x} +
 \|z/r^2\|_{L^2_x})^2 \lec \|q\|_{L^\I_2}^2.}
For $m \geq 2$, $\psi \in L^2_\I$, and so 
\EQ{
  \| \ip{ S q }{ \psi^s/s } \|_{L^1_t} \lec \| S \|_{L^2_t L^2_x} \| q \|_{L^2_t L^\I_2} \lec \de\|q\|_{L^2_t L^\I_2},}
either by \eqref{S L^2 diss} or \eqref{S L^2 disp}. 
The last two terms of~\eqref{eq mu} are bounded in $L^1_t$ by
\EQ{
  \|s \dot \mu\|_{L^2_t} \|z/r\|_{L^2_t L^\I_x} \|\rs{(r|\fy|+r^2|\fy_r|)}\|_{L^1_x} \lec \|q\|_{L^2_t L^\I_2}^2,}
where we used \eqref{muL2} and \eqref{z bd q}. 

Thus we have obtained
\EQ{
  \| \dot\mu - e^{i\ti\al}\ip{q_t }{ \psi^s/s} \|_{L^1_t} 
  \lec \de \|q\|_{L^2_t L^\I_2}. }
Integrating by parts in $t$, the leading term is rewritten as 
\EQ{ \label{pint t}
  e^{i\ti\al}\ip{q_t }{ \psi^s/s}
 \pt= \p_t \ip{e^{i\ti\al}q}{\psi^s/s} 
 - i(s\dot\al + sS(t,0))\ip{ e^{i\ti\al}q }{\rs{\psi}} 
 \pr + \dot s\ip{e^{i\ti\al}q }{ (r\p_r+1)\rs{\psi}}.}
The last term can be bounded in $L^1_t$ by using \eqref{muL2}, 
\EQ{
  \|s\dot\mu\|_{L^2_t} \|\ip{q }{ (r\p_r+1)\rs\psi}\|_{L^2_t} 
   \lec \|q\|_{L^2_t L^\I_2} \|\ip{q }{ (r\p_r+1)\rs\psi}\|_{L^2_t}.}
If $m=2$ or $m>3$, then $(r\p_r+1)\psi\in L^1$, and so the above is further bounded by 
$\|q\|_{L^2_t L^\I_2}^2$. 
When $m=3$, we need some extra effort to bound the last factor in 
$L^2_t$ -- this is done in the next section.  

If $m>2$, we have for the leading term 
\EQ{
  |\ip{e^{i\ti\al}q }{ \psi^s/s}| 
  \lec \|q\|_{L^2_x} \|\psi\|_{L^2_x} \lec \|q(0)\|_{L^2_x}, }
while for $m=2$ this term can be infinite from the beginning. 
We will show in Section~\ref{m=2} that the time difference 
$[\ip{e^{i\ti\al}q}{\psi^s/s}]_0^t$ can be controlled for finite $t$,
but still may become unbounded as $t \to \I$ for some initial data.

This also means that the second last term of \eqref{pint t} is
beyond our control when $m=2$, and so in this case we force it to 
vanish by making the assumptions $a_2=0$ and $v_2=0$.   
For the other cases ($m > 2$), we should estimate $S(t,0)$, 
for which we use in the Schr\"odinger case ($a=i$) that 
\EQ{
 S =  Q - \int_r^\I \frac{2Q}{r} dr, \pq Q=\frac{1}{2}|q|^2 + \frac{m}{r}w_3 = O(|q|^2+|z/r|^2+qh_1^s/r),}
since $|w_3|\lec |q||\nu|$ and $|\nu| \lec h_1^s + |z|$. 
Thus we get at each $t$, using \eqref{z bd q},  
\EQ{
 \|S\|_{L^\I_x} \lec \|q\|_{L^\I_2}^2 + s^{-1}\|q\|_{L^\I}.}
Then the second term in \eqref{pint t} is bounded in $L^1_t$ 
\EQ{
 \pt\|q\|_{L^2_t L^\I_2}^2\|\ip{q}{\psi^s/s}\|_{L^\I_t} + (\|s\dot\al\|_{L^2_t}+\|q\|_{L^2_tL^\I_x})\|\ip{q}{\rs\psi}\|_{L^2_t}  
 \pr\lec \|q\|_{L^2_t L^\I_2}(\de+\|\ip{q}{\rs\psi}\|_{L^2_t}),}
where we used \eqref{muL2}. 
If $m>3$, then $\psi\in L^1$ and hence the last factor 
$\|\ip{q}{\rs\psi}\|_{L^2_t}$ is bounded by 
$\|q\|_{L^2_t L^\I_2}\lec\de$. 
Its estimate for $m=3$ is deferred to the next section. 

In the dissipative case $a_1>0$, we estimate simply by \eqref{qeq} at each $t$  
\EQ{ \label{S diss est}
 \|S\|_{L^\I_x} \pt\lec (\|q/r\|_{L^2_x} + \|z/r^2\|_{L^2_x} + \|h_1^s/r^2\|_{L^2_x})\|L_{\M{v}}^*q\|_{L^2_x},}
and hence the second term in $L^1_t$ is bounded by 
\EQ{
 \pt\|q\|_{L^2_t L^\I_2}\|q\|_{L^2_t X} + (\|s\dot\mu\|_{L^2_t}+\|q\|_{L^2_t X})\|\ip{q}{\rs\psi}\|_{L^2_t}
 \pr\lec \|q\|_{L^2_t L^\I_2}(\de+\|\ip{q}{\rs\psi}\|_{L^2_t}),}
where we used \eqref{muL2} and \eqref{q L^2 diss}. 

Thus we have obtained all the necessary estimates to prove
Theorem~\ref{thm1} when $m>3$. 
Furthermore, its proof for $m=3$ will be complete once we show
\EQ{ \label{rest for m=3}
 \|\ip{q}{\rs\psi}\|_{L^2_t} + \|\ip{q}{(r\p_r+1)\rs\psi}\|_{L^2_t}
 \lec \de,}
which will be done in Section~\ref{m=3}.

For Theorem \ref{thm2}, it remains to derive the asymptotic formula 
\eqref{s asy} from the leading term $\ip{q}{\psi^s/s}$, and to show 
that all of the asymptotic behavior (1)--(6) can be realized by choice 
of the initial data $\M{u}(0,x)$ -- this is done in Section~\ref{m=2}.

\section{Special estimates for $m=3$} \label{m=3}
In this section we finish the proof of Theorem \ref{thm1} 
by showing \eqref{rest for m=3}. 
It suffices to estimate the leading term for $r\to\I$: 
\EQ{ \label{L2 inner prod}
  \|\ip{r^{-2}\chi(r) }{ q}\|_{L^2_t} \lec \de,}
with $\chi\in C^\I$ satisfying $\chi(r)=0$ for $r<1$ and $\chi(r)=1$ for $r>2$, 
since the rest decays at slowest $O(r^{-8})\in L^1_x$, for which we can simply 
use $q\in L^2_t L^\I_x$. Once the above is proved, we can conclude that 
\EQ{
  \|\mu\|_{L^\I_t} \lec |\mu(0)| + \|q\|_{L^\I_tL^2_x\cap L^2_tL^\I_2}. }
The boundedness of $\mu$ and the scattering of $q$ imply that
the ``normal form'' correction $\ip{e^{i \tilde{\alpha}} q }{ \psi^s/s}$
converges to zero, and so $\mu(t)$ is convergent as $t\to\I$. 

To estimate \eqref{L2 inner prod}, we use perturbation from the free 
evolution $e^{at\De_2^{(2)}}$: 
\EQ{
 \dot q - a\De_2^{(2)} q = N_0 + N_2,}
where $N_0$ and $N_2$ are as in \eqref{def N0} and \eqref{def N1N2}, satisfying 
\EQ{
  N_0 \in r^{-2}\LR{r/s}^{-4}L^2_tL^\I_x, \pq N_2 \in L^1_tL^2_x. }

For the contribution of $N_2$ as well as the initial data, 
we use the following estimate. 
\begin{lem} \label{free L^2_t/r^2}
For any $l>0$, any $a\in\C^\times$ with $\Re a\ge 0$, and any 
functions $g(r),f(r)$, and $F(t,r)$, we have 
\EQ{
  \pt \|\ip{g }{ e^{at\De_2^{(l)}}f}\|_{L^2_t(0,\I)}  
   \lec \|r^2g\|_{L^\I_x}\|f\|_{L^2_x}, 
  \pr \|\ip{g}{\int_{-\I}^t e^{a(t-s)\De_2^{(l)}}F(s)ds}\|_{L^2_t(\R)}
   \lec \|r^2g\|_{L^\I_x}\|F\|_{L^1_t L^2_x}.}
\end{lem}
\begin{proof}
We start with the estimate for the free part. Let $\hat g=\F_l g$ and $\hat{f}=\F_l f$. 
The above $L^2_t$ norm equals by Plancherel in space,
\EQ{ \label{Plan x}
  \|\ip{\hat g}{ e^{-atr^2}\hat{f}}\|_{L^2_t(0,\I)} \sim \|\int_0^\I e^{-at\s}G(\s)d\s\|_{L^2_t(0,\I)},}
where we put
\EQ{
 G(\s):=\hat g(\s^2)\bar{\hat f(\s^2)}.}
If $a_1>0$, then \eqref{Plan x} is bounded by Minkowski 
\EQ{
 \pt\le \|\int_0^\I e^{-a_1t\s}|G(\s)|d\s\|_{L^2_t(0,\I)}
  \pn\le \int_0^\I e^{-a_1\s} \|t^{-1}G(\s/t)\|_{L^2_t(0,\I)} d\s
  \pr\le \|G\|_{L^2_\s(0,\I)} \int_0^\I \s^{-1/2}e^{-a_1\s}d\s 
  \lec \|G\|_{L^2_\s(0,\I)}.}
If $a_1=0$, then $a_2\not=0$ and \eqref{Plan x} is bounded by Plancherel in $t$, 
\EQ{
 \le \|\int_0^\I e^{-ia_2t\s}G(\s)ds\|_{L^2_t(\R)} \sim \|G\|_{L^2_\s(0,\I)}.}
Thus in both cases we obtain  
\EQ{
   \|\ip{g }{ e^{at\De_2^{(l)}}f}\|_{L^2_t(0,\I)} 
   \pt\lec \|G(\s)\|_{L^2_\s(\R)} \lec \|\hat g\|_{L^\I}\|\hat f\|_{L^2} \sim \|\hat g\|_{L^\I} \|f\|_{L^2_x}.}
Then the first desired estimate follows from 
\EQ{
 |\hat g(\r)| \le \int_0^\I |J_l(r\r)||g(r)|rdr \le \|r^2g\|_{L^\I_x} \int_0^\I |J_l(r)|\frac{dr}{r} \sim \|r^2g\|_{L^\I_x},}
since $|J_l(r)|\lec\min(r^l,r^{-1/2})$ for $r>0$. 

By duality, the estimate on the Duhamel term is equivalent to 
\EQ{
 \|\int_0^\I \la(s+t)e^{\bar{a}s\De_2^{(l)}}g(x)ds\|_{L^\I_t L^2_x}
 \lec \|r^2g\|_{L^\I_x} \|\la\|_{L^2_t}, }
which is equivalent to 
\EQ{
 \|\int_0^\I \la(t)e^{\bar{a}t\De_2^{(l)}}g(x)dt\|_{L^2_x} \lec \|r^2g\|_{L^\I_x} \|\la\|_{L^2_t},}
which is dual to the first estimate. 
\end{proof} 

For the potential part $N_0$, we transfer the equation to 
$\R^6$ by $u=r^{-2}q$ and consider
\EQ{
  u_t -a \De_6^{(0)} u = r^{-2}N_0. }
Then thanks to the decay of the potential, we have 
\EQ{
  r^{-2}N_0 \in L^2_t L^{10/7}_x (\R^6) 
  \subset L^2_t \dot H^{-1/5}_{3/2}(\R^6),}
as long as $s(t)$ is away from $0$ and $\I$. 
Then by the endpoint Strichartz or the energy estimate on $\R^6$, 
the corresponding Duhamel term is bounded in 
$L^2_t \dot H^{-1/5}_3(\R^6)$, and since 
$|\na_x r^{-4}\chi(r)|\lec r^{-5}$, we have 
$r^{-4}\chi \in \dot H^1_{5/4}(\R^6)\subset \dot H^{1/5}_{3/2}(\R^6)$. 
Thus to summarize, we have
\EQ{
  \|\ip{r^{-2}\chi }{ q}\|_{L^2_t} \lec 
  \|q(0)\|_{L^2_x} + \|q\|_{L^2_t L^\I_{2,x}}.}

{\bf Completion of the proof of Theorem~\ref{thm1}:}
Let initial data $\M{u}(0)$ be specified as in Theorem~\ref{thm1}.
The existence of a unique local-in-time solution $\M{u}(t)$ in 
the given spaces can be deduced by working in the $(\mu, q)$ 
variables (using the bijection of Lemma~\ref{coord lem}) and 
using estimates similar to those of Sections~\ref{decay} 
and~\ref{parameter}. The details are carried out in the 
Schr\"odinger case ($a=i$) in~\cite{GKT2}, and carry over to 
the general case in a straightforward way (in fact, there are 
well-established methods for energy-space local existence 
in the dissipative case, starting with the pioneering 
work~\cite{Str} on the heat-flow). It follows from
this local theory that the solution continues as long as 
$\mu(t)$ is bounded and $q$ is bounded in 
$L^\I_t L^2_x \cap L^2_t L^\I_2$.  
 
For $m > 3$, the estimates of the previous four sections 
give the boundedness of $q$ and $\mu$ which ensure the solution
is global, as well as the convergence of $\mu(t)$.
The convergence to a harmonic map then follows from      
the estimates of Section~\ref{decay}.
$\Box$

\section{Special estimates for $m=2$, $a>0$, $v_2=0$} \label{m=2}

Let $m=2$ and (with no further loss of generality) $a=1$. 
By the bijective correspondence $\M{v}\leftrightarrow(\mu,q)$, it is 
clear that $v_2=0$ is equivalent to $\mu,q\in\R$. 
It remains to control the leading term for the parameter dynamics   
\EQ{
 \ip{q }{ \psi^s/s}. }
In particular, we will show that this can diverge to $\pm\I$, or 
oscillate between them for certain initial data. 

First by the asymptotics for $r\to\I$, we have 
$\psi+cr_{1<}^{-1}\in L^2_x$, where we denote
\EQ{
 r_{a<}^{-1} = \CAS{ r^{-1} &(r>a) \\ 0 &(r\le a)}
 \pq r_{<b}^{-1} = \CAS{ r^{-1} &(r<b) \\ 0 &(r\ge b)}
 \pq r_{a<b}^{-1} = \CAS{ r^{-1} &(a<r<b) \\ 0 &\text{(otherwise)}}}
Hence we may replace $s^{-1}\psi^s$ by $-cr_{s<}^{-1}$ modulo $o(1)L^\I_t$. 

Next we want to replace $q$ by the free solution 
$q^0:=e^{t\De_2^{(1)}}q(0)$. 
For that we use the following pointwise estimate to bound $q - q_0$: 
\begin{lem} \label{pt bd fund sol}
Let $\Re a>0$ and $l\ge 1$. Then for any function $g(r)$ satisfying $|g(r)|\le \LR{r}^{-1}$, we have 
\EQ{
 |e^{at\De_2^{(l)}}g(r)| \lec \min(1,1/r,r/t).}
\end{lem}

\begin{proof} 
Let $a_1=\Re a$. By using the explicit kernel we have 
\EQ{ \label{free sol g}
  e^{2at\De_2^{(l)}}g(r) \pt = \frac{C}{t}\int_0^\I\int_{-\pi}^{\pi}
  e^{-a\frac{r^2-2rq\cos\th+q^2}{t}+il\th} g(q)q d\th dq, }
and the integral in $\th$ can be rewritten by partial integration on
$e^{i\th}$ as 
\EQ{
  \int_{-\pi}^\pi \frac{rq}{ilt}\sin\th
  e^{-a\frac{r^2-2rq\cos\th+q^2}{t}+il\th} g(q)q d\th dq. }
The double integral for $|q-r|>r/2$ is bounded by using the second
form by
\EQ{
 \int_0^\I \frac{rq}{t^2} e^{-a_1\frac{r^2+q^2}{4t}} \min(q,1)dq \lec
r e^{-a_1\frac{r^2}{4t}}\min(t^{-1/2},t^{-1}) \lec \min(1,r/t,1/r),} 
and that for $|q-r|<r/2$ is bounded by using the first form by
\EQ{
  t^{-1}\int_0^\I \int_{-\pi}^{\pi} e^{-a_1\frac{(r-q)^2}{t}}
  e^{-a_1\frac{r^2 \th^2}{8t}} d\th \min(r,1)dq
  \pt\lec t^{-1} t^{1/2} (t/r^2)^{1/2} \min(r,1) \pr\lec \min(1,1/r), }
and by the second form by $\lec r^2/t \times 1/r = r/t$.
\end{proof}

The nonlinear part of $q$ contributes as 
\EQ{ \label{q nonlin}
  \ip{r_{s<}^{-1} }{ q-q^0} = 
  -\int_0^t \ip{e^{\De_2^{(m-1)}(t-t')}r_{s(t)<}^{-1}}{V(t')q(t')} dt',}
where the potential term is given by 
\EQ{
 V = \frac{2m(1-v_3)}{r^2} + \frac{m}{r} w_3
   = \frac{2m(1-h_3^s)}{r^2} + O(\res v_3/r^2) + O( q/r ).}
The contribution from the last two parts is estimated with the 
$r^{-1}$ bound from the above Lemma, thus bounded by  
\EQ{
  (\|\res v_3/r^2\|_{L^2_{t,x}} + 
  \| q/r \|_{L^2_{t,x}} )\|q/r\|_{L^2_{t,x}} \lec \|q(0)\|_{L^2_x}^2. }

We need to be more careful to estimate the other term $q(1-h_3^s)/r^2$.
First by Schwarz and the pointwise estimate, we have
\EQ{
  \pt\left|\int_0^t (e^{a\De_2^{(m-1)}(t-t')}r_{s(t)<}^{-1} \; | \; 
  r^{-2}q(1-h_3^s)) dt'\right|^2 \pr \le\|q/r\|_{L^2_{t,x}}^2 \int_0^t 
  \int_0^\I \min(s(t)^{-1},r^{-1},r/(t-t'))^2 \LR{r/s(t')}^{-4m} 
  \frac{dr}{r} dt', }
where we also used that $|1-h_3(r)|\lec \LR{r}^{-2m}$. 
It suffices to bound the last double integral. 
Let $\t=t-t'$. For $0<\t<s(t)^2$, the $r$ integral is bounded by
\EQ{
  \pt \int_0^{\t/s(t)}\frac{r^2}{\t^2}\frac{dr}{r} + 
  \int_{\t/s(t)}^{s(t)} \frac{1}{s(t)^2}\frac{dr}{r} 
  + \int_{s(t)}^\I \frac{1}{r^2}\frac{dr}{r}
  \pn \lec s(t)^{-2}(1 + \log(s(t)^2/\t)),}
hence its $\t$ integral is bounded by
\EQ{
  \int_0^{s(t)^2} s(t)^{-2} (1+\log(s(t)^2/\t)) d\t 
  = 1 + \int_0^1 |\log\th|d\th < \I. }
For $s(t)^2<\t$, the $r$ integral is bounded by
\EQ{
 \int_0^{s(t')} \frac{r^2}{\t^2} \frac{dr}{r} + \int_{s(t')}^\I \frac{r^2 s(t')^{4m}}{\t^2 r^{4m}} \frac{dr}{r}
  \lec \frac{s(t')^2}{\t^2},}
and its $\t$ integral is bounded by square of 
\EQ{
  \|s(t')/\t\|_{L^2_\t(s(t)^2,t)} \pt \lec 
  \|s(t)/\t\|_{L^2_\t(s(t)^2,t)} + 
  \|(s(t)-s(t-\t))/\t\|_{L^2_\t(s(t)^2,t)}
  \pr\lec 1 + \int_0^1 \|\dot s(t-\th\t)\|_{L^2_\t(s(t)^2,t)}d\th
  \pr\lec 1 + \int_0^1 \|\dot s\|_{L^2}\th^{-1/2} d\th \lec 1. }
Thus we obtain
\EQ{
  |\ip{r_{s(t)<}^{-1} }{ q-q^0}| \pt\lec 
  \|q/r\|_{L^2_{t,x}}(\|\res v_3/r^2\|_{L^2_{t,x}} 
  + \|\dot s\|_{L^2_t} + 1) \pn\lec \|q(0)\|_{L^2},}
namely we may replace $q$ by the free solution $q^0$ in the leading 
asymptotic term. 

Furthermore, we can freeze the scaling parameter because
\EQ{
 \pn|\ip{r_{s(t)<}^{-1}-r_{s(0)<}^{-1} }{ q^0}|
  \pt\lec \|r_{s(t)<}^{-1}-r_{s(0)<}^{-1}\|_{L^2} \|q^0(t)\|_{L^2}
 \pr\lec |[\log s]_0^t|^{1/2}o(1) 
 \lec o(1)(|[\log s]_0^t| + 1).}
Thus we obtain 
\EQ{
 (1+o(1))[2\log s]_0^t = -c[\ip{r_{s(0)<}^{-1}}{q^0}]_0^t + O(1), \pq(t\to\I)}
where $O(1)$ is convergent. 

The leading term is further rewritten in the Fourier space by using that 
\EQ{
 \F_1[r_{s(0)<}^{-1}](\r) \pt= \r^{-1}\int_{s(0)\r}^\I J_1(r)dr 
 \pr= \r^{-1}J_0(s(0)\r) = \r_{<1/s(0)}^{-1} + s(0)R(s(0)\r),\pq \exists R\in L^2_x.}
Let $\hat q_0:=\F_1 q(0)$. By Plancherel we have $\|q(0)\|_{L^2_x}=\|\hat q_0\|_{L^2_x}$ and 
\EQ{ 
 [-\ip{r_{s(0)<}^{-1}}{q^0}]_0^t 
 \pt= \ip{(1-e^{-t r^2})r_{<1/s(0)}^{-1}}{\hat q_0} + O(1)
 \pr= \ip{r^{-1}_{1/\sqrt{t}<1/s(0)}}{ \hat q_0} + O(1)
 \pr=\ip{\F_1 r^{-1}_{1/\sqrt{t}<1/s(0)}}{ q(0)} + O(1)
 \pr= 2\pi \int_{s(0)}^{\sqrt{t}} q(0,r)dr + O(1).}

Thus we obtain (using that $c=\|h_1\|_{L^2_x}^{-2}=\pi^{-2}$),
\EQ{ \label{asy in q}
 (1+o(1))[\log s]_0^t = \frac{1}{\pi}\int_{s(0)}^{\sqrt{t}} q(0,r)dr + O(1),}
and the error term $O(1)$ converges to a finite value as $t\to\I$.  

{\bf Completion of the proof of Theorem~\ref{thm2}:}
As in the proof of Theorem~\ref{thm1}, we now have all the estimates
to conclude the solution is global (in particular, $\mu(t)$ remains
finite by the above formula and estimates), and the convergence to
the harmonic map family follows from the estimates of
Section~\ref{decay}. It remains to consider the asymptotics of $s(t)$. 
 
Since $q(0)\in L^2_x$ does not require $\int_1^\I|q(0,r)|dr<\I$, 
it is easy to make up $q(0)\in L^2$, for any given $s(0)\in(0,\I)$, 
such that the first term on the right of \eqref{asy in q} attains arbitrarily given 
$\limsup\ge\liminf\in[-\I,\I]$ as $t\to\I$. 
In particular, all of the asymptotic behaviors (1)-(6) in 
Theorem~\ref{thm2} can be realized by appropriate choices of  
$(q(0),s(0))$, for which Lemma~\ref{coord lem} ensures existence of 
corresponding initial data $\M{u}(0)\in\Si_2$. 

Using that $v_2=0$, we can further rewrite the leading term 
in terms of $\M{v}$. Since $\V{e}=(v_3,i,-v_1)$, we have
\EQ{
 q = \M{w} \vdt \V{e} 
   = v_1v_{3r}-v_3v_{1r} + \frac{2v_1}{r} 
   = -\be_r + \frac{2v_1}{r},}
where $\be$ is defined by $\M{v}=(\cos\be,0,\sin\be)$. 
Hence we have 
\EQ{
 (1+o(1))[\log s]_0^t 
    = \frac{2}{\pi} \int_{s(0)}^{\sqrt{t}} \frac{v_1(0,r)}{r} dr 
      + O(1),}
where $O(1)$ converges as $t\to\I$. 

$\Box$

\section{Proofs of the key linear estimates} \label{pf lin ests}

\subsection{Uniform bound on the right inverse $R_\fy$}
\begin{proof}[Proof of Lemma \ref{bd on R}]
Let $s=1$ and omit it. It suffices to prove
\EQ{ \label{sup bound R}
 \pt \|R_\fy g\|_{r^\th L^\I} \lec \|\fy\|_{r^{-\th}L^1}\|g\|_{r^{\th+1}L^1_\I}, 
 \pr \|R_\fy^* f\|_{r^{-\th-1}L^\I} \lec \|\fy\|_{r^{-\th}L^1_\I} \|f\|_{r^{-\th}L^1_\I}.}
From this we get by duality, 
\EQ{
 \pt \|R_\fy g\|_{r^{\th}L^\I_1} \lec \|\fy\|_{r^{-\th}L^1_\I} \|g\|_{r^{\th+1}L^1},  
 \pr \|R_\fy^* f\|_{r^{-\th-1}L^\I_1} \lec \|\fy\|_{r^{-\th}L^1} \|f\|_{r^{-\th}L^1},}
and the bilinear complex interpolation covers the intermediate cases.  

It remains to prove \eqref{sup bound R}. 
We rewrite the kernel of $R_\fy$ 
\EQ{
 R_\fy g = \iint \frac{h_1(r)}{h_1(r'')}
  \chi(r,r',r'')\bar\fy(r')h_1(r')r' g(r'') dr'' dr',}
where $\chi(r)$ is defined by
\EQ{
 \chi(r,r',r'')=\CAS{1 &(r'<r''<r),\\ -1 &(r<r''<r'),\\ 0 &\text{(otherwise).}} }
We decompose the double integral dyadically such that $r\sim 2^j$, 
$r'' \sim 2^k$ and $r' \sim 2^l$, 
and let 
\EQ{
 \pt A_j = 2^{-\th j}\|R_\fy g\|_{L^\I(r\sim 2^j)}, \pq B_j = 2^{(\th+1)j}\|R_\fy^*f\|_{L^\I(r\sim 2^j)}, 
 \pr \fy_l = 2^{\th l}\|\fy\|_{L^1(r\sim 2^l)}, \pq g_k = 2^{(-\th-1)k}\|g\|_{L^1(r\sim 2^k)}, 
 \pq f_k = 2^{\th k}\|f\|_{L^1(r\sim 2^k)}.}

For $R_\fy$, we have
\EQ{
 A_j \lec \sum_{\substack{j-1\le k\le l+1 \\ l-1 \le k \le j+1}} 2^{-m|j|-\th j+m|k|+\th k - m|l|-\th l}\fy_l g_k.}
The sums over $k$ are bounded for $j-1\le k\le l+1$ and for $l-1\le k\le j+1$ respectively by
\EQ{
 2^{-m|j|-\th j - m|l|-\th l}\fy_l \sup_k g_k \times\CAS{\max(2^{(-m+\th)j},1) \max(2^{(m+\th)l},1), \\
 \max(2^{(-m+\th)l},1)\max(2^{(m+\th)j},1),}}
and since the exponential factors are bounded, after summation over $l$ we get 
\EQ{
 \|R_\fy g\|_{r^{\th} L^\I} \lec \sum_l \sup_k \fy_l g_k,}
as desired. 
For $R_\fy^*$ in \eqref{sup bound R}, we have
\EQ{
 B_j \lec \sum_{\substack{k-1\le j\le l+1 \\ l-1 \le j \le k+1}} 2^{m|j|+\th j-m|k|-\th k - m|l|-\th l}\fy_l f_k.}
Then the sums over $k$ and $l$ are bounded in both cases by
\EQ{
 2^{m|j|+\th j}\sup_{k,l}\fy_l f_k \min(2^{mj-\th j},1) \min(2^{-mj-\th j},1),}
and hence $\|R_\fy^*f\|_{r^{-\th-1} L^\I} \lec \sup_{l,k} \fy_l g_k$, as desired. 

Next we show the optimality. 
Let $b\in\Z$, and choose any $g$ which is piecewise constant on each dyadic interval $(2^j,2^{j+1})$, $\supp g\subset[2^b,\I)$, and $g\ge 0$. Then for $0<r\le 2^b$ we have
\EQ{
 R_\fy g(r) \pt= h(r)\int_{2^b}^\I \int_{2^b}^a h(s)^{-1}\fy(a)h(a)a g(s)ds da
 \pr\gec h(r) \sum_{j\ge b} \sum_{k=b}^{j-1} 2^{m|k|+\th k} g_k 2^{-m|j|-\th j} \fy_j
 \gec h(r) \sum_{j\ge b} g_{j-1}\fy_j,}
where we denote $g_k=\|g\|_{r^{\th-1}L^\I(r\sim 2^k)}$ and $\fy_j=\|\fy\|_{r^{-\th}L^1(r\sim 2^j)}$. 
Choosing a test function $\psi\in C_0^\I(0,\I)$ satisfying $\psi\ge 0$, $\supp\psi\subset(0,2^b)$ and $\ip{h}{\psi}=1$, we see that $\fy_j\in\ell^{p'}_j(j>b)$ is necessary since we can choose arbitrary non-negative $g_k\in\ell^p_k(k>b)$. 
Similarly by choosing $\supp g\subset(0,2^b]$ and $\supp\psi\subset(2^b,\I)$, we see that $\fy_j\in\ell^{p'}_j(j<b)$ is also necessary.  
\end{proof}

\subsection{Double endpoint Strichartz estimate}
\label{end}

Lemma \ref{Str H} holds for more general radial potentials. 
We call 
\EQ{ \label{eq:KS}
   \left\| r^{-1}\int_0^t e^{i(t-s)H} f(s) ds 
   \right\|_{L^2_{t,x}} \lec \|rf\|_{L^2_{t,x}}}
the Kato estimate for the operator $H$, and 
\EQ{ \label{eq:DE}
  \|u\|_{L^2_t(L^\I_2)} \lec \|u(0)\|_{L^2_x} + \|iu_t+Hu\|_{L^2_t(L^1_2)}.}
the double endpoint Strichartz estimate for $H$. 
Lemma \ref{Str H} is a consequence of the following. 

\begin{thm}\label{thm:endpoint}
For any $m>0$, the double endpoint Strichartz \eqref{eq:DE} holds for radially symmetric $u(t,x)=u(t,|x|)$ and  
$H=\De_2^{(m)}=\p_r^2+r^{-1}\p_r-m^2r^{-2}$.  
\end{thm}

\begin{lem} \label{thm:perturb}
Suppose $H_0$ and $H=H_0+V$ are both self-adjoint on 
$L^2(\R^2)$ and $|x|^2 V(x) \in L^\I(\R^2)$. 
Assume that the Kato estimate \eqref{eq:KS} holds for $H$, 
and that the double endpoint Strichartz estimate \eqref{eq:DE} holds for $H_0$. 
Then we have the double endpoint Strichartz also for $H$. 
The same is true when we restrict all functions to radially symmetric ones, if $V$ is also symmetric. 
\end{lem}

\begin{cor}\label{cor}
Let $V = V(|x|) \in C^1(\R^2 \backslash \{ 0 \})$ 
be a radially-symmetric function with $|x|^2 V \in L^\infty(\R^2)$,  
and suppose $H = -\Delta + V$ is self-adjoint on $L^2(\R^2)$.
Let $f(t,x) = f(t,|x|)$ be radial. Then
\begin{enumerate}
\item
the Kato estimate~\eqref{eq:KS} holds for $H$ if and only
if the double-endpoint estimate~\eqref{eq:DE} holds for $H$,
\item 
both estimates hold provided
\begin{equation}
\label{eq:repulsive}
  \inf_{r > 0} r^2 V(r) > 0, \quad
  \inf_{r > 0} -r^2 (r V(r))_r > 0.
\end{equation}
\end{enumerate}
\end{cor}
Since our linearized operator $H^s$ satisfies \eqref{eq:repulsive}, the above implies Lemma \ref{Str H}. 
\begin{proof}[Proof of Corollary \ref{cor}]
The first statement follows directly from
Theorem~\ref{thm:endpoint} and Lemma~\ref{thm:perturb}.
For the second statement:
the methods of~\cite{BPST}, adapted to the 2-dimensional
radial setting (detailed in~\cite{GKT2}), imply 
that conditions~\eqref{eq:repulsive} yield 
the resolvent estimate~\eqref{eq:resolvent}, hence
the Kato estimate, and the double endpoint estimate.
\end{proof}

\begin{rem}
While the double-endpoint estimate~\eqref{eq:DE} always implies the 
Kato estimate~\eqref{eq:KS}, the 
reverse implication does not hold in general. 
For example, consider $H = -\Delta$ acting on 
$2D$ functions with zero angular average.
The Kato estimate in this case can be verified,
for example, by using the methods of~\cite{BPST} to 
establish the resolvent estimate
\begin{equation}
\label{eq:resolvent}
  \sup_{\la\not=0}\|(H-\la)^{-1}\fy\|_{L^{2,-1}_x} 
  \lec \|\fy\|_{L^{2,1}_x},
\end{equation}
from which the Kato estimate follows by Plancherel in $t$
(see~\cite{GKT2} for details). 
On the other hand, if the double-endpoint estimate were to 
hold for zero-angular-average functions, so would
the endpoint homogeneous estimate. Since the 
latter is known to hold for radial functions (see Tao \cite{Tao}), 
it would therefore hold for all $2D$ functions, which is 
false (see Montgomery-Smith \cite{MT}, also see \cite{Tao}). 
Alternatively, a constructive counter-example is given 
by placing delta functions of the same mass but opposite sign 
at $(1,0)$ and $(0,1)$ in the plane. 
\end{rem}

\begin{proof}[Proof of Theorem~\ref{thm:endpoint}] 
Following \cite{Nrl}, we use the identity
\[
 \iint_{s<t} F(s,t) dsdt
 = C \int_0^\I \frac{dr}{r} \int_\R \frac{da}{r} \int_{a-3r}^{a-r} ds
 \int_{a+r}^{a+3r} dt F(s,t),
\]
for the decomposition, where $C>0$ is some explicit positive constant. 
Define the bilinear operators $I_j$ for $j\in\Z$ by 
\[
 I_j(f,g) := \int_{2^j}^{2^{j+1}} \frac{dr}{r} \int_\R \frac{da}{r}
 \int_{a-3r}^{a-r} ds \int_{a+r}^{a+3r} dt
 \PX{e^{i(t-s)\De_2^{(m)}}f(s)}{g(t)}
\]
where $f$ and $g$ are radial (i.e. $f(s) = f(|x|,s)$, etc.).

The desired estimate follows from 
\[
 \sum_{j\in\Z}|I_j(f,g)| \lec \|f\|_{L^2_t L^1_x} \|g\|_{L^2_t L^1_x}.
\]

Using the $1/t$ decay for 
$\| e^{i t \De_2^{(m)}} \|_{L^1 \to L^\I}$, 
we can easily bound the supremum of the summand. 
To get summability, we need decay both faster and slower than $1/t$. 
In fact we have, for $\fy = \fy(|x|)$ radial,
\EQ{ \label{decay est}
 \|e^{it\De_2^{(m)}}\fy\|_{L^{\I,\mu}_x} \lec
 |t|^{-1+\mu}\|\fy\|_{L^{1,-\mu}_x}, \quad -m\le\mu\le 1/2.}
This follows easily from the explicit fundamental solution 
\[
  (e^{it\De_2^{(m)}} \fy)(r) = \frac{c_m}{t}
  \int_0^\infty e^{i(r^2+\rho^2)/4t}
   J_m \left( \frac{r\rho}{2t} \right) \phi(\rho) \rho d\rho
\] 
($c_m$ a constant) in terms of the Bessel function $J_m$ of the first kind, for which
\[
  \sup_{s > 0} s^\mu J_m(s) < \infty,
  \quad -m \leq \mu \leq 1/2.
\]
\Del{The above estimate can be easily obtained by change of variables. 
First, it is trivial for $m=\mu=0$. 
For $-\mu=m\in\N/2$, the function $v(t,r):=r^{-m}e^{it\De_2^{(m)}}\fy$ satisfies the free Schr\"odinger on $\R^{1+d}$ with $d=2m+2$. 
For $\mu=m=1/2$, the function $v(t,r):=r^{1/2}e^{it\De^{1/2}}\fy$ satisfies the free Schr\"odinger on $\R$. 
In particular, we get the above estimates for $m=1/2$. 
For larger $m\in\N/2$, we should get the same estimates, because the equation for those $v$ have just some potentials like $+C/r^2$. 
For general $m\ge 0$, we should get the estimate by the complex interpolation.}

Next, when the decay is slower than $1/t$, namely if we choose $\mu>0$
in \eqref{decay est}, then we get a non-endpoint Strichartz 
estimate using the Hardy-Littlewood-Sobolev inequality in time:
\EQ{ \label{nonendpoint}
 \| \int_\R e^{-is\De_2^{(m)}} f(s) ds \|_{L^2_x}
 \lec \|f\|_{L^{p'}_t L^{1,-\al}_x},}
for $0<\al\le 1/2$ and  $1/p=1/2-\al/2$. 

Now the rest of the proof follows along the lines of Keel-Tao \cite{KT}. 
We will prove that
\EQ{ \label{piecewise est} 
 \pt |I_j(f,g)| \lec 2^{j(\al+\be)/2} \|f\|_{L^2_t L^{1,-\al}_x} \|g\|_{L^2_t L^{1,-\be}_x},}
for
\EQ{ \label{exp1}
 -m\le \al=\be <0}
and for 
\EQ{ \label{exp2}
 0 < \al, \be \le 1/2.}
For the first exponents \eqref{exp1}, we use the decay estimate
\eqref{decay est} and the $L^{\I,\al}-L^{1,-\al}$ duality at each $(s,t)$. Then we get
\[
\begin{split}
 |I_j(f,g)| & \lec \int_{2^j}^{2^{j+1}} \frac{dr}{r} \int_\R
 \frac{da}{r} \int_{a-3r}^{a-r} ds \int_{a+r}^{a+3r} dt
 \frac{\|f(s)\|_{L^{1,-\al}_x} \|g(t)\|_{L^{1,-\al}_x}}{|t-s|^{1-\al}} \\
 & \lec \int_{2^j}^{2^{j+1}} \frac{dr}{r} \int_\R \frac{da}{r}
 2^{j\al} \|f\|_{L^2_t(a-3r,a-r;L^{1,-\al}_x)}
 \|g\|_{L^2_t(a+r,a+3r;L^{1,-\al}_x)} \\
 & \lec \int_{2^j}^{2^{j+1}} \frac{dr}{r} 2^{j\al}\frac{1}{r}
 \|f\|_{L^2_{a,t}(-3r<t-a<-r;L^{1,-\al}_x)}
 \|g\|_{L^2_{a,t}(r<t-a<3r;L^{1,-\al}_x)} \\
 & \lec 2^{j\al}\|f\|_{L^2_t L^{1,-\al}_x} \|g\|_{L^2_t
 L^{1,-\al}_x},
\end{split}
\]
where we used H\"older for $s,t,a$. 

For the second exponents \eqref{exp2}, we use the non-endpoint
Strichartz \eqref{nonendpoint} for both integrals in $s$ and $t$,
after applying the Schwartz inequality in $x$. Then we get
\[
\begin{split}
 |I_j(f,g)| & \lec  \int_{2^j}^{2^{j+1}} \frac{dr}{r} \int_\R
 \frac{da}{r} \|f\|_{L^{p'}_t(a-3r,a-r;L^{1,-\beta}_x)}
 \|g\|_{L^{q'}_t(a+r,a+3r;L^{1,-\al}_x)} \\
 & \lec \int_{2^j}^{2^{j+1}} \frac{dr}{r} \int_\R \frac{da}{r}
 2^{j(\al+\be)/2}\|f\|_{L^2_t(a-3r,a-r;L^{1,-\beta}_x)}
 \|g\|_{L^2_t(a+r,a+3r;L^{1,-\al}_x)},
\end{split}
\]
and the rest is the same as above, where 
$1/p'=1/2+\al/2$ and $1/q'=1/2+\beta/2$. 
Thus we get \eqref{piecewise est} both for \eqref{exp1} and
\eqref{exp2}. 
By bilinear complex interpolation (cf. \cite{BL}), we can extend the 
region $(\al,\be)$ to the convex hull:
\EQ{ \label{convhul}
 \al>\frac{m}{m+\frac{1}{2}} \left(\be-\frac{1}{2} \right), 
 \pq \be>\frac{m}{m+\frac{1}{2}} \left(\al-\frac{1}{2} \right), 
 \pq \al,\be<\frac{1}{2}.}
The only property we need is that this set includes a neighborhood 
of $(0,0)$, where we are looking for the summability. 

Now we use bilinear interpolation 
(see \cite[Exercise 3.13.5(b)]{BL} and \cite{ON})
\[
\begin{split}
 & T:X_i\times X_j \to Y_{i+j} \pq (i,j,i+j\in\{0,1\}) \\
 & \implies T:X_{\th_1,r_1}\times X_{\th_2,r_2} \to Y_{\th_1+\th_2,r_0}
 \pq 1/r_0=1/r_1+1/r_2,
\end{split}
\]
where $X_{\th,r}:=(X_0,X_1)_{\th,r}$ denotes the real interpolation space. 

The above bound \eqref{piecewise est} can be written as
\[
 \|I(f,g)\|_{\ell^{-(\al+\be)/2}_\I} \lec \|f\|_{L^2_t L^{1,-\al}_x}
 \|g\|_{L^2_t L^{1,-\be}_x},
\]
where $\ell_p^{\al}$ denotes the weighted space over $\Z$:
\[
 \|a\|_{\ell_p^{\al}} := \|2^{j\al}a_j\|_{\ell^p_j(\Z)}.
\]
Hence the bilinear interpolation implies that
\EQ{ \label{sum est}
 \|I(f,g)\|_{\ell_1^{-(\al+\be)/2}} \lec \|f\|_{L^2_t L^{1,-\al}_2} \|g\|_{L^2_t L^{1,-\be}_2},}
for all $(\al,\be)$ in~\eqref{convhul}, where 
\[
 \|\fy\|_{L^{p,s}_q}^q := \sum_{k\in\Z} \|2^{ks}\fy\|_{L^p(|x|\sim
 2^k)}^q,
\]
and we used the interpolation property of weighted spaces (cf. \cite{BL}): 
\[
\begin{split}
 & (\ell_\I^{\al},\ell_\I^{\be})_{\th,q} = \ell_q^{(1-\th)\al+\th\be},
  \pq \al\not=\be, \\
  & (L^{p,\al},L^{p,\be})_{\th,q} = L^{p,(1-\th)\al+\th\be}_q,
  \pq \al\not=\be.
\end{split}
\]
By choosing $\al=\be=0$ in \eqref{sum est}, we get the desired result.
\end{proof}

\begin{proof}[Proof of Lemma \ref{thm:perturb}]
By time translation, we can replace the interval 
of integration in \eqref{eq:KS} and \eqref{eq:DE} by $(-\I,t)$. 
Then by taking the dual, we can also replace it by $(t,\I)$. Adding those two, we can replace it by $\R$. 
Then the standard $TT^*$ argument implies that
\[
 \|e^{iHt}\fy\|_{L^2_t(L^{2,-1})} \lec \|\fy\|_{L^2},
 \quad \|e^{iH_0t}\fy\|_{L^2_t(L^\I_2)} \lec \|\fy\|_{L^2}.
\]
Now let
\[
 u = \int_{-\I}^t e^{i(t-s)H} f(s) ds.
\]
Then the Duhamel formula for the equation
\[
 iu_t + H_0 u = f-Vu
\]
implies that
\[
 u = \int_{-\I}^t e^{i(t-s)H_0}(f-Vu)(s) ds.
\]
Applying \eqref{eq:KS} for $H$ and \eqref{eq:DE} for $H_0$, 
and using $L^{2,1}\subset L^1_2$, we get
\EQ{ \label{Kato to D.E}
 \|u\|_{L^2_t(L^\I_2)} \pt \lec \|f-Vu\|_{L^2_t(L^1_2)}
 \pr \lec \|f\|_{L^2_t(L^1_2)} + \|r^2 V\|_{L^\I_x} \|u\|_{L^2_t(L^{2,-1})}
 \pn \lec \|f\|_{L^2_t(L^{2,1})}.}
Then by duality we also get
\[
 \|u\|_{L^2_t(L^{2,-1})} \lec \|f\|_{L^2_t(L^1_2)}.
\]
Feeding this back into \eqref{Kato to D.E}, we get
\[
 \|u\|_{L^2_t(L^\I_2)} \lec \|f\|_{L^2_t(L^1_2)}.
\]
The estimate on $e^{iHt}\fy$ is simpler, or can be derived from the
above by the $TT^*$ argument.
\end{proof}


\appendix
\section{Landau-Lifshitz maps from $\S^2$} 
The same stability problem on $\S^2$, instead of $\R^2$, is much easier in the dissipative case, because the eigenfunctions get additional decay from the curved metric on $\S^2$. Indeed we have convergence for all $m\ge 1$: 
\begin{thm}
Let $m\ge 2$, $a\in\C$ and $\Re a> 0$. 
Then there exists $\de>0$ such that for any
$\M{u}(0,x)\in\Si_m$ with $\E(\M{u}(0))\le 4m\pi+\de^2$, 
we have a unique global solution $\M{u}\in C([0,\I);\Si_m)$
satisfying $\na\M{u}\in L^2_{t,loc}([0,\I);L^\I_x)$. 
Moreover, for some $\mu_\I \in \C$ we have 
\EQ{
 \|\M{u}(t)-e^{m\th R}\M{h}[\mu_\I]\|_{L^\I_x} + \E(\M{u}(t)-e^{m\th R}\M{h}[\mu_\I]) \to 0 \pq(t\to\I).}
\end{thm}
Our proof does not give a uniform bound on $\de$ if
$m=1$, but we have 
\begin{thm}
Let $m=1$, $a\in\C$, $\Re a> 0$ and $\mu_0\in\C$. 
Then there exists $\de>0$ such that for any
$\M{u}(0,x)\in\Si_1$ with $\E(\M{u}(0)-\M{h}[\mu_0])\le \de^2$, 
we have a unique global solution $\M{u}\in C([0,\I);\Si_1)$
satisfying $\na\M{u}\in L^2_{t,loc}([0,\I);L^\I_x)$. 
Moreover, for some $\mu_\I \in \C$ we have 
\EQ{
 \|\M{u}(t)-e^{m\th R}\M{h}[\mu_\I]\|_{L^\I_x} + \E(\M{u}(t)-e^{m\th R}\M{h}[\mu_\I]) \to 0 \pq(t\to\I).}
\end{thm}

The proof is essentially a small subset of that in the $\R^2$ case, so we just indicate necessary modifications. 
\begin{proof}[Outline of Proof]
By the stereographic projection, we can translate the problem to $\R^2$ with the metric $g(x)dx^2$, where $g(x)=(1+r^2/4)^{-2}$. 
The harmonic maps are the same, while the evolution equation is changed to 
\EQ{
 \pt q_t = i S q - a L_{\M{v}} g^{-1} L_{\M{v}}^* q,
 \pq S = \int_\I^r g^{-1}(q+\frac{m}{r}\nu)\cdt ia L_{\M{v}}^*q dr.}
In this setting we can use the ``standard" orthogonality to decide $\mu$:
\EQ{ \label{orth}
 0 = \ip{z}{gh_1^s},}
since $gh_1\in\LR{r}^{-1}L^1$. 
The energy identity 
\EQ{
 \p_t\|q\|_{L^2_x}^2 = -2a_1 \ip{g^{-1}L_{\M{v}}^*q}{L_{\M{v}}^*q}}
implies the a priori bound on $q$:
\EQ{
 \|q\|_{L^\I_t L^2_x} + \|g^{-1/2}L_{\M{v}}^*q\|_{L^2_t L^2_x} \lec \|q(0)\|_{L^2_x} \sim \de.}
Since $g^{-1/2}\gec r$, we get (by using $q=R_\fy^{s*}L^{s*}q$ as on $\R^2$), 
\EQ{
 \|q\|_{L^2_x} \lec \|rL_{\M{v}}^*q\|_{L^2_x} \lec \|g^{-1/2}L_{\M{v}}^*q\|_{L^2_x} \in L^2_t.}
Then by the orthogonality \eqref{orth} we have $z=R_{\phi_s}^s L^sz$ with  
\EQ{
 \phi_s := s^2g(rs)h_1/\ip{gh_1^s}{h_1^s},}
and so
\EQ{
 \|z/r\|_{L^2_x} \lec \|L^s z\|_{L^2_x} \|\phi_s\|_{L^1_2}.}
Since 
\EQ{
 \int_0^\I g(rs)\min(r,1/r)^{-m} rdr \sim \CAS{ \min(1,s^{-2}) &(m>2) \\ \min(s^{-1},s^{-2}) &(m=1)},}
we have 
\EQ{
 \|\phi_s\|_{L^1_2} \sim \CAS{ 1 & (m\ge 2) \\ \max(s^{-1},1) &(m=1)}.} 
Anyway, if $\de$ is small enough (depending on $s$), we get by the same argument as on $\R^2$, 
\EQ{
 \|z\|_X \lec \|q\|_{L^2_x} \in L^2_t \cap L^\I_t.} 

Differentiating the orthogonality, we get
\EQ{ \label{eq mu S^2}
 \dot\mu\ip{h_1^s}{gh_1^s} \pt= -\ip{\Sc{M}aL_{\M{v}}^*q}{h_1^s} - \ip{gz}{(\frac{\dot\mu_1}{m}r\p_r+i\dot\mu_2h_3^s)h_1^s}
 \pr=-\ip{(\Sc{M}_r+\frac{m\res{v}_3}{r}\Sc{M})aq}{h_1^s} 
 \prQ- \ip{gz}{\{\frac{\dot\mu_1}{m}(r\p_r+m)+i\dot\mu_2(h_3^s-1)\}h_1^s},}
where on the second equality we used that $L^sh_1^s=0$ and $\ip{gz}{h_1^s}=0$. 
Using that
\EQ{
 |(r\p_r +m)h_1|+|(h_3-1)h_1| \lec \min(r^{m-1},r^{-3m-1}) \lec \LR{r}^{-4},}
we can bound the last term in \eqref{eq mu S^2} by
\EQ{
 |\dot\mu|\|z\|_{L^\I_x}\min(s^2,1),}
which is much smaller than the term on the left. 
The second last term in \eqref{eq mu S^2} is bounded at each $t$ by 
\EQ{
 \|q\|_{L^2_x}^2 \|h_1^s\|_{L^\I_x} \lec \|q\|_{L^2_x}^2.}
If $m>1$, we can improve this for $s<1$ as follows. 
By the same argument as on $\R^2$, we have
\EQ{
 \|q\|_{L^\I_{2,x}} \lec \|L_{\M{v}}^*q\|_{L^2_x} \lec \|g^{-1/2}L_{\M{v}}^*q\|_{L^2_x},}
where we need $m>1$ for the boundedness of $R_{\phi_s}^s:r^2L^1_2\to rL^\I_2$ and $R_\fy^{s*}:rL^1_2\to L^\I_2$. 
Then we can replace the above estimate in the region $r<1$ by 
\EQ{
 \|q\|_{L^\I_2\cap L^2_x} \|h_1^s\|_{L^1_\I+L^\I_x} \lec \|q\|_{L^2_x}^2\min(s^2,1).}
Thus we obtain
\EQ{
 \|\dot\mu\|_{L^1_t} \lec \CAS{\de^2 &(m\ge 2) \\ C(s)\de^2 &(m=1)},}
and hence if $\de>0$ is small enough, we get the desired convergence as on $\R^2$. 
\end{proof} 

\section*{Acknowledgments} 
This work was conducted while the second author was visiting the
University of British Columbia, by the support of the 21st century COE
program, and also by the support of the Kyoto University Foundation.
The research of Gustafson and Tsai is partly supported by NSERC grants
no. 251124-07 and 261356-08. The research of Nakanishi was partly 
supported by the JSPS grant no.~15740086.


\bigskip

\noindent{Stephen Gustafson},  gustaf@math.ubc.ca \\
Department of Mathematics, University of British Columbia, 
Vancouver, BC V6T 1Z2, Canada

\bigskip

\noindent{Kenji Nakanishi},
n-kenji@math.kyoto-u.ac.jp \\
Department of Mathematics, Kyoto University,
Kyoto 606-8502, Japan

\bigskip

\noindent{Tai-Peng Tsai},  ttsai@math.ubc.ca \\
Department of Mathematics, University of British Columbia, 
Vancouver, BC V6T 1Z2, Canada

\end{document}